\documentclass[11pt,reqno]{amsart}
\usepackage{latexsym,amsfonts,amssymb,amsthm,url,amsbsy,amscd}
\usepackage[english]{babel}
\usepackage[latin1]{inputenc}
\usepackage{enumerate}
\usepackage{amsmath}
\numberwithin{equation}{section}
\newtheorem{theorem}{Theorem}[section]
\newtheorem{proposition}{Proposition}[section]
\newtheorem{corollary}{Corollary}[section]
\newtheorem{lemma}{Lemma}[section]

\usepackage{color}
\definecolor{darkred}{rgb}{0.8,0,0}
\definecolor{darkblue}{rgb}{0,0,0.7}
\definecolor{darkgreen}{rgb}{0,0.4,0}

\newcommand{\ee}{{\rm e}}
\newcommand{\eps}{\varepsilon}
\newcommand{\id}{{\rm id}}
\newcommand{\dd}{\;{\rm d}}
\newcommand{\RR}{{\mathbb R}}
\newcommand{\R}{{\mathbb R}}
\newcommand{\NN}{{\mathbb N}}
\newcommand{\LL}{{\rm L}}
\renewcommand{\L}{{\rm L}}
\newcommand{\HH}{{\rm H}}
\renewcommand{\H}{{\rm H}}
\newcommand{\PP}{{\mathcal P}}

\newcommand{\CC}{{\mathcal C}}

\newcommand{\C}{{\mathcal C}}
\newcommand{\FF}{{\mathcal F}}
\newcommand{\EE}{{\mathcal E}}
\newcommand{\KK}{{\mathcal K}}
\newcommand{\TT}{{\mathcal T}}
\newcommand{\DD}{{\mathcal D}}
\newcommand{\V}{{\mathcal V}}
\renewcommand{\ln}{\log}
\newcommand{\norm}[1]{\left \| #1 \right \|}

\newcommand{\un}{{\rm 1\kern -2.5pt l}}

\begin{document}
\title{A hybrid variational principle for the Keller-Segel system in $\R^2$}
\date{\today}
\begin{abstract}
We construct weak global in time solutions to the classical Keller-Segel system  cell movement by chemotaxis in two dimensions when the total mass is below the well-known critical value. Our construction takes advantage of the fact that the Keller-Segel system can be realized as a gradient flow in a suitable functional product space. This allows us to  employ a hybrid variational principle which is a generalisation of the minimising implicit scheme for Wasserstein distances introduced by Jordan, Kinderlehrer and Otto (1998).
\end{abstract}
\author{Adrien Blanchet}
\address{Adrien Blanchet -- TSE (GREMAQ, Universit\'e Toulouse 1 Capitole), 21 All\'ee de Brienne, F-31015 Toulouse cedex 6}
\email{Adrien.Blanchet@ut-capitole.fr}

\author{Jos\'e Antonio Carrillo}
\address{Jos\'e Antonio Carrillo -- Department of Mathematics, Imperial College London, London SW7 2AZ, UK}
\email{carrillo@imperial.ac.uk}

\author{David Kinderlehrer}
\address{David Kinderlehrer -- Department of Mathematical Sciences, Carnegie Mellon University, Pittsburgh, PA 15213, USA}
\email{davidk@andrew.cmu.edu}

\author{Micha{\l } Kowalczyk}
\address{Micha{\l } Kowalczyk -- Departamento de Ingenier\'{\i}a Matem\'atica and Centro
de Modelamiento Matem\'atico (UMI 2807 CNRS), Universidad de Chile, Casilla
170 Correo 3, Santiago, Chile.}
\email {kowalczy@dim.uchile.cl}

\author{Philippe Lauren{\c{c}}ot}
\address{Philippe Lauren{\c{c}}ot -- Institut de Math\'ematiques de Toulouse, UMR~5219, Universit\'e de Toulouse, CNRS, F--31062 Toulouse Cedex 9, France}
\email{Philippe.Laurencot@math.univ-toulouse.fr}

\author{Stefano Lisini}
\address{Stefano Lisini -- Universit\`a degli Studi di Pavia, Dipartimento di Matematica "F. Casorati", via Ferrata 1, I-27100 Pavia, Italy}
\email{stefano.lisini@unipv.it}
\pagestyle{plain}


\maketitle
\tableofcontents
\section{Introduction}
\subsection{The model}
The parabolic-parabolic Keller-Segel model \cite{Keller-Segel-70, KS2} is a drift-diffusion system given by
\begin{equation}\label{ks original}
\left\{
  \begin{array}{l}
{\partial_t u}=  \Delta u -\chi_0 {\rm div} \left[   u\nabla v\right]\,, \vspace{.2cm}\\
\tau_0 {\partial _tv } = D_0\Delta v - \alpha_0\, v + \beta_0\, u\,,\vspace{.2cm}\\
u_0 \in \LL^1_+(\RR^2)\, ,\quad v_0 \in \HH^1(\RR^2),
  \end{array}
\quad (t,x) \in (0,\infty) \times \RR^2\,,
\right.
\end{equation}
where $\chi_0$, $\tau_0$, $D_0$, $\alpha_0$, and $\beta_0$ are given positive parameters and $\LL^1_+(\R^2)$ denotes the positive cone of $\LL^1(\R^2)$. The system (\ref{ks original}) is a widely accepted model of chemotaxis, a phenomenon in which organisms, most notably dictyostelium discoideum, with density $u$, are attracted by a chemo-attractant $v$, produced by them. This feedback mechanism may lead to an aggregation phenomena expressed by the concentration of the distribution function $u$ at some points and it may even grow  without bounds as time progresses leading to blow-up in density. The Keller-Segel model, which looks simple at first sight, is a very rich mathematical system and it has been an object of very extensive investigation for the last  forty years. By introducing the new unknown functions
\begin{equation*}
  \rho(t,x):=\frac{u(t,x)}{\|u_0\|_{1}}\,, \quad \phi(t,x):=\frac{D_0}{\beta_0\|u_0\|_{1}} v(t,x)\,,
\end{equation*}
and the two initial data
\begin{equation*}
  \rho_0:=\frac{u_0}{\|u_0\|_{1}}, \quad \phi_0:=\frac{D_0}{\beta_0\|u_0\|_{1}} v_0,
\end{equation*}
where $\|\cdot\|_1$ denotes the $\L^1$-norm, we obtain the equivalent system
\begin{equation}\label{num2}
\left\{
  \begin{array}{l}
{\partial_t\rho}=  \Delta \rho -\chi {\rm div} \left[   \rho\nabla \phi\right]\,, \vspace{.2cm}\\
\tau {\partial_t \phi} = \Delta \phi - \alpha\, \phi +  \rho\,,\vspace{.2cm}\\
\rho_0 \in \LL^1_+(\RR^2) \,, \quad \phi_0 \in \HH^1(\RR^2)\,,
  \end{array}
\quad (t,x) \in (0,\infty) \times \RR^2\,,
\right.
\end{equation}
with
\begin{equation*}
  \tau:=\frac{\tau_0}{D_0}\,,\quad \alpha:=\frac{\alpha_0}{D_0}\,,\quad\mbox{and}\quad \chi:=\frac{\beta_0\,\chi_0\|u_0\|_{1}}{D_0}\;.
\end{equation*}
Note that with this rescaling $\rho_0$ is a probability density.
It is immediate to notice that the total mass of $\rho$ is formally preserved along the flow,
\begin{equation*}
\int_{\mathbb R^2}\rho(t,x)\dd x=\int_{\mathbb R^2}\rho_0(x)\dd x=1,\qquad t\ge 0
\end{equation*}
and that $\rho(t,\cdot)\geq 0$ if $\rho_0\geq 0$. Thus we can reduce to construct solutions such that $\rho(t,\cdot)$ is a probability density for every $t>0$. We stress that the solutions obtained with our technique automatically enjoy this property.

Taking $\tau=0$ we obtain the so called parabolic-elliptic Keller-Segel model. Although our focus here is the case $\tau>0$, it is instructive to revise some basic facts about this ``simplified'' system. Taking the initial condition $\rho_0$ such that the second moment
\[
\int_{\R^2} |x|^2 \rho_0(x)\dd x<\infty,
\]
and calculating formally the time derivative of the 2-moment $\mathcal M_2(t)=\int_{\R^2} |x|^2 \rho(t,x)\dd x$ we obtain ${\dd}\mathcal M_2(t)/{\dd t} <0$ provided that $\chi>8\pi$. This means that at some finite time $T>0$,  $\mathcal M_2(T)= 0$ which would imply total concentration of the mass. The conclusion is that, for $\chi>8\pi$, there is finite time blow-up of classical solutions for the  parabolic-elliptic  Keller-Segel model \cite{bilerhilhorst}.  It turns out that when $\chi<8\pi$ solutions exist and are bounded for all times \cite{MR2226917}. The borderline case $\chi=8 \pi$ was considered in \cite{BCM} where it was shown that solutions for initial data with finite second moment exist globally but they become unbounded and converge to  a Dirac delta function as $t\to \infty$. In all these references, solutions were constructed by approximation methods leading to free energy solutions. 

However, one can use the gradient flow approach introduced in \cite{MR1617171,MR1842429} for diffusions and in \cite{Carrillo-McCann-Villani03} for nonlocal interactions to the parabolic-elliptic Keller-Segel model as in~\cite{BCC08,BCC12}. In particular, the gradient flow interpretation leads to a nice understanding of the energy landscape in the critical mass case $\chi=8 \pi$. There are infinitely many stationary solutions, all of them locally asymptotically stable, for which a second Liapunov functional was found in \cite{BCC12}. The key property of the gradient flow interpretation is that all stationary solutions are infinitely apart from each other in the optimal transport euclidean distance, and each of them has its own basin of attraction.

Returning to the parabolic-parabolic model, it is known that under the condition $\chi < 8\pi$ and with reasonable assumptions on the initial condition, solutions to~\eqref{num2} exist for all times~\cite{CC08,Mixx}. Our objective is to give another proof of the global in time existence. This proof, which  is  based on the so called {\it hybrid variational principle}, does not give strictly speaking any new existence result. Our objective is to emphasize an important, and not immediately apparent,  property: the variational character  of the Keller-Segel model. It also sheds some light on why when $\chi>8\pi$  the issue of global existence versus blow up is so delicate in the parabolic-parabolic case. We note that it is proven in \cite{bilerdolbeault} that when $\chi> 8\pi$ and $\tau$ is sufficiently large then there exist global self-similar solutions. It has also been shown recently in \cite{BGK} that for any initial condition and $\chi>8\pi$ there exists $\tau$ such that the Keller-Segel model has a global solution with this initial condition (the Cauchy problem being understood  in some weak sense, and solutions are not necessarily unique). It is also proven recently in \cite{Sch1,Sch2} that blow-up solutions exist for supercritical mass close to critical and the blow-up profile has been characterised, see also~\cite{lushnikov} for a formal analysis and~\cite{herrerovelazquez} for related results in a bounded domain. This variational interpretation of the Keller-Segel model suggests that there might be a path in the function space along which the free energy functional becomes unbounded leading possibly to blow-up ``along'' this path. For numerical simulations inspired from the scheme see~\cite{epsh,epsh2}. 

Finally, let us mention that the solutions are proven to be unique and the functional has some convexity over the set of solutions as soon as the cell density becomes bounded \cite{CLM}. 

\subsection{The formal gradient flow interpretation}
We denote by $\PP(\RR^2)$ the set of Borel probability measures on $\R^2$ with finite second moment,
and by
\begin{equation*}
  \KK:=\{\rho \in \PP(\RR^2)\,:\, \rho \ll \dd x \mbox{ and } \int_{\RR^2} \rho \log \rho \,\dd x < \infty\}.
\end{equation*}
Let us define the free energy of the Keller-Segel system \eqref{num2} as $\EE:\PP(\R^2)\times \L^2(\R^2)\to(-\infty,+\infty]$ by
\begin{equation}\label{eq:functional}
\EE[\rho,\phi]= \int_{\RR^2} \left\{\frac{1}{\chi}\rho(x)\log\rho(x)-\rho(x)\,\phi(x) + \frac12\,|\nabla \phi(x)|^2+\frac{\alpha}{2}\,\phi(x)^2 \right\}\dd x \;,
\end{equation}
if $(\rho,\phi)\in\KK\times \H^1(\R^2)$
and $\EE[\rho,\phi]= +\infty$ otherwise.
We will see in Lemma \ref{tintin} that if $\chi<8\pi$ then $\EE$ cannot reach the value $-\infty$.
The domain of $\EE$ coincides with $\KK\times \H^1(\R^2)$.

We observe that, at least formally, the system \eqref{num2} has the following ``gradient flow'' structure
\begin{equation}\label{eq:KS3}
\left\{\begin{array}{l}
    \partial_t \rho = \chi\nabla \cdot \left(\rho \nabla \dfrac{\delta \EE}{\delta \rho}\right),\vspace{.3cm}\\
 \tau   \partial_t \phi = -\dfrac{\delta \EE}{\delta \phi},
    \end{array}\right.
\end{equation}
where ${\delta \EE}/{\delta \rho}$ and ${\delta \EE}/{\delta \phi}$ denote the  first variation of the functional $\EE$ with respect to the variables $\rho$ and $\phi$ respectively. Indeed, the right hand side in the first equation of \eqref{eq:KS3} is, up to a factor $\chi$, the ``gradient'' of $\EE$ along the curve $t\mapsto \rho(t,\cdot)$ with respect to the Kantorovich-Rubinstein-Wasserstein distance, referred to hereafter as the Wasserstein distance, $ {\mathcal W}_2= d_W$ and ${\delta \EE}/{\delta \phi}$ is the ``gradient'' of $\EE$ along the curve $t\mapsto \phi(t,\cdot)$ with respect to the $\LL^2(\R^2)$ distance. We can formally compute the dissipation of $\EE[\rho,\phi]$ in \eqref{eq:functional} along a solution of \eqref{eq:KS3} as
$$
\begin{aligned}
\dfrac{\dd }{\dd t} \EE[\rho(t),\phi(t)] &= \int_{\RR^2} \left[ \dfrac{\delta \EE}{\delta \rho} \partial_t \rho  + \dfrac{\delta \EE}{\delta \phi} \partial_t \phi \right] \dd x  \\
&= -\chi \int_{\RR^2} \left| \nabla \dfrac{\delta \EE}{\delta \rho}\right|^2 \rho \,\dd x  - \dfrac1{\tau} \int_{\RR^2} \left(\dfrac{\delta \EE}{\delta \phi} \right)^2 \dd x\,.
\end{aligned}
$$
We recall here that the variational scheme introduced by Jordan, Kinderlehrer, and Otto in~\cite{MR1617171} is based, generally speaking, on a gradient flow of some free energy in the Wasserstein topology.  Here, we work in a product space topology $\mathcal P(\R^2)\times \L^2(\R^2)$ and this justifies the name hybrid variational principle for the implicit scheme that we will introduce in what follows, and which makes the notion of the gradient flow in this context rigorous.

We should point out that hybrid variational principles have been already used to show existence of solutions for  a model of the Janossy effect in a dye doped liquid crystal \cite{MR2461815}, for the  the Keller-Segel model with critical diffusion in $\R^N$, $N\geq 3$~\cite{bl2xx} and some of its variants~\cite{zinsl,zinslma,mimura}, and for the thin film Muskat problem in \cite{LMxx}. 

\subsection{Main results}
When a problem has a gradient flow structure, then its trajectories follow the steepest descent path, and one way to prove existence of solutions is by employing some implicit discrete in time approximation scheme. Suitable time interpolation and compactness arguments should finally give the convergence towards a solution when the time step goes to zero.

The main result of this work is the construction of solutions to~\eqref{num2} using an adapted version of the implicit variational schemes introduced in \cite{MR1617171} for the case of the Wasserstein distance. 
The general setting in metric spaces was also developed by De Giorgi school with the name of minimizing movement approximation scheme (see \cite{AGS08} and the references therein).  The minimising scheme is as follows: given an initial condition $(\rho_0,\phi_0)\in\KK \times \HH^1(\RR^2)$ and a time step $h>0$, we define a sequence $(\rho_{h}^n,\phi_{h}^n)_{n\ge 0}$ in $\KK\times \HH^1(\RR^2)$ by
\begin{equation}\label{scheme:jko}
\left\{
  \begin{array}{l}
    (\rho_{h}^0,\phi_{h}^0) = (\rho_0,\phi_0)\,, \vspace{.3cm}\\
\displaystyle(\rho_{h}^{n+1},\phi_{h}^{n+1})  \in {\rm Argmin}_{(\rho,\phi)\in\KK\times \H^1(\R^2)}  \,\mathcal{F}_{h,n}[\rho,\phi]\,, \qquad n\ge 0\,,
  \end{array}
\right.
\end{equation}
where
\begin{equation*}
   \FF_{h,n}[\rho,\phi] := \frac{1}{2 h}\ \left[ \frac{1}{\chi}d_W^2(\rho,\rho_{h}^{n}) + \tau\ \|\phi-\phi_{h}^{n}\|_{{\rm L}^2(\R^2)}^2 \right] + \EE[\rho,\phi]\,.
\end{equation*}
and $d_W$ denotes the Wasserstein distance which is defined in Section~\ref{wasser}.
\begin{theorem}[Convergence of the scheme]\label{main1}
Assume that the constants in the Keller-Segel system satisfy $0<\chi<8\pi$, $\tau>0$ and $\alpha> 0$.
Given $(\rho_0, \phi_0)\in \mathcal K\times \H^1(\R^2)$ there exists a  sequence $(\rho_h^n, \phi_h^n)\in \mathcal K\times \H^1(\R^2)$
satisfying the variational principle \eqref{scheme:jko}.
Defining the piecewise constant function
\[
(\rho_h(t),\phi_h(t))=(\rho_h^n,\phi_h^n), \quad \mbox{ if } t\in ((n-1)h, nh],
\]
there exists a decreasing sequence $(h_j)_j$ going to $0$ as $j$ goes to $\infty$
and a continuous curve
$(\rho, \phi):[0,\infty)\to \PP(\R^2)\times \L^2(\R^2)$
such that
\[
\begin{aligned}
\rho_{h_j}(t)&\rightharpoonup \rho(t)\quad  \mbox{weakly in } \LL^1(\R^2), \qquad  t>0,\\
\phi_{h_j}(t)&\rightharpoonup \phi(t)\quad  \mbox{weakly in } \L^2(\R^2) \mbox{ and strongly in } \L_{\rm loc}^p(\R^2), \quad  t>0, \quad  p\in [1,+\infty).
\end{aligned}
\]
Moreover, $(\rho,\phi)$ satisfies the following regularity properties:
\begin{itemize}
\item[(i)] $\rho\in \C^{1/2}([0,T];\PP(\R^2))$ and $\phi\in \C^{1/2}([0,T];\L^2(\R^2))$, for every $T>0$.
\item[(ii)] For all $T>0$, we have the estimate
\begin{equation}\label{ME}
\begin{aligned}
& \sup_{t\in[0,T]}\left\{\int_{\R^2}\left[|x|^2\rho(t,x)+\rho(t,x)|\log\rho(t,x)| \right]\dd x  + \|\phi(t)\|_{\HH^1(\R^2)}\right\}<\infty \,.
\end{aligned}
\end{equation}
\item[(iii)] The pair $(\rho,\phi)$ is a weak solution of the Keller-Segel system \eqref{num2} in the sense  that
\begin{equation} \label{weak solution}
\begin{aligned}
&\int_0^{+\infty}\int_{\RR^2} \partial_t\xi\rho - \nabla\xi\cdot \left( \nabla \rho - \chi\rho\nabla\phi \right)  \dd x \dd t=0, \quad \text{for all }\xi\in \CC^\infty_0((0,\infty)\times\R^2),\\
&\quad \tau \partial_t\phi=\Delta \phi-\alpha\phi+\rho, \quad\mbox{a.e. in }\ (0,\infty)\times \R^2.
\end{aligned}
\end{equation}
\end{itemize}
\end{theorem}

\begin{theorem}[Dissipation inequality]\label{main2}
  Under the same assumptions as Theorem~\ref{main1}, the solution $(\rho, \phi)$ constructed in Theorem~\ref{main1} furthermore satisfies: 
  \begin{itemize}
\item[(i)] The regularity property: for all $T>0$, $\rho\in \L^2(0,T;\R^2)\cap\L^1((0,T);W^{1,1}(\R^2))$, $\phi\in\L^2((0,T);\HH^2(\R^2))\cap \HH^1((0,T);\L^2(\R^2))$, and the Fisher information bound holds
\begin{equation}\label{B}
\int_0^T\int_{\R^2}\left|\frac{\nabla\rho}{\rho}\right|^2\rho\dd x\dd t <+\infty\,.
\end{equation}
\item[(ii)] The energy dissipation inequality: For all $T > 0$
\begin{equation}
\begin{aligned}\label{energy inequality}
\frac{1}{\chi}\int_0^T\int_{\R^2}\left|\frac{\nabla\rho}{\rho}-\chi\nabla\phi \right|^2\,\rho\dd x\dd t
&+ \frac1\tau\int_0^T\|\Delta\phi -\alpha\phi+\rho \|_{{\rm L}^2(\R^2)}^2 \dd t +\\
&+\EE[\rho(T),\phi(T)]
\leq \EE[\rho_0,\phi_0].
\end{aligned}
\end{equation}
  \end{itemize}
\end{theorem}
Obviously the interval  $(0,T)$ can be replaced by any  $(T_1,T_2)$  in  $\R^+$. Since we do not know whether the non-negative $\rho$ is positive in $\R^2$ the meaning in~\eqref{B} is that the integrand is equal to $\left|{\nabla \log \rho}\right|^2\rho$ if $\rho$ is positive and zero elsewhere. Owing to the energy dissipation an alternative formulation for the equation on $\rho$ is $\partial_t \rho + \nabla \cdot \left(\rho J\right)=0$ with $J:=\nabla \log \rho - \chi \nabla \phi$ where $\rho$ is positive and $J(t) \in \LL^2(\R^2,\rho(t)\dd x)$ for almost every $t$.
\medskip

Several difficulties arise in the proof of the well-posedness and convergence of the minimising scheme.  First of all, since the energy $\EE$ is not displacement convex, standard results from~\cite{AGS08,Vi03} do not apply and even the existence of a minimiser is not clear. This is primarily because we choose to work in the whole space $\R^2$ rather than a bounded domain, a choice made to replicate the optimal known results. Section~\ref{mini} is devoted to this minimisation problem. Let us mention that this functional has some convexity properties but only when restricted to bounded densities as proven in \cite{CLM}. However, we cannot take advantage of this convexity for the construction of weak solutions with the regularity stated on the initial data.

The second issue has to do with the regularity of the minimisers obtained in each step without which we cannot show convergence of the discrete scheme to a solution of \eqref{weak solution}. To derive the Euler-Lagrange equation satisfied by a minimiser $({\rho},{\phi})$ of $\mathcal{F}_{h,n}$ in $\KK\times \HH^1(\RR^2)$, the parameters $h$ and $n$ being fixed, we consider an ``optimal transport'' perturbation for ${\rho}$ and a $\LL^2$-perturbation for ${\phi}$ defined for $\delta\in (0,1)$ by
\begin{equation*}
 \rho_\delta= (\id + \delta\,\zeta) _\# {\rho}\;, \quad \phi_\delta := {\phi} + \delta\, w \;,
\end{equation*}
where $\zeta\in\C^\infty_0(\RR^2;\RR^2)$ and $w\in\C^\infty_0(\RR^2)$. We note that $\rho_\delta$ is the push forward of $\rho$ by the map $\id+\delta\zeta$.  Identifying the Euler-Lagrange equation requires passage to the limit as $\delta\to 0$ in
\begin{equation*}
\frac{d^2_W(\rho_\delta ,\rho_{0})-d^2_W({\rho} ,\rho_{0})}{2\delta} \;\;\mbox{ and }\;\;
\frac{1}{\delta}\int_{\RR^2} (\rho_\delta \log \rho_\delta-{\rho}\log{\rho})\dd x,
\end{equation*}
which can be done by standard arguments, and also in
\begin{equation*}
\frac{1}{\delta}\ \int_{\RR^2} ({\rho}\,{\phi}-\rho_\delta\,\phi_\delta)(x) \dd x
= \int_{\RR^2} {\rho}(x) \left[ \frac{{\phi}(x)-{\phi}(x+\delta\zeta(x))}{\delta} - w(x+\delta\zeta(x)) \right]\dd x\;.
\end{equation*}
This is where the main difficulty lies: indeed, since ${\phi}\in \HH^1(\RR^2)$, we only have
\begin{equation*}
\frac{{\phi}\!\circ\!(\id + \delta\zeta) -{\phi}}{\delta} \rightharpoonup \zeta \cdot \nabla{\phi}\quad \mbox{in $\LL^2(\RR^2)$,}
\end{equation*}
while ${\rho}$ is only in $\KK$.
Consequently the product ${\rho} \zeta\cdot\nabla{\phi}$ which is the candidate for the limit
may not be well defined and the regularity of $({\rho},{\phi})$ has to be improved. We also remark that the dissipation of the functional involves $\Delta \phi$, and therefore we need to show additional regularity on the potential $\phi$ to have a well-defined dissipation of~\eqref{eq:functional}.
A general strategy to overcome this regularity issue is explained in subsection \ref{sec:MMS} using an adaptation of the arguments in~\cite{bl2xx,MMS09}.


\section{Preliminaries}
\subsection{Lower semi-continuity of functional defined on measures }
The following lower semi-continuity result is very useful.
For the proof we refer to Theorem 2.34 and Example 2.36 of \cite{AFP00}.
\begin{proposition}\label{prop:lscIF}
 Let $(\mu_n)_{n\ge0}$, $(\gamma_n)_{n\ge 0}$ be two sequences of Borel positive measures in $\R^d$, $d \ge 1$, such that $\mu_n$ is absolutely continuous with respect to $\gamma_n$ for each $n \ge 1$. Consider $f:\R \to [0,\infty]$ a convex function with super-linear growth at infinity.
Assume that $(\mu_n)_{n\ge0}$, $(\gamma_n)_{n\ge 0}$ weakly-* converge (in duality with $C_c(\R^d)$) to $\mu$ and $\gamma$ respectively and
\begin{equation*}
    \sup_{n \ge 1}\int_{\R^d} f\Big(\frac{\dd\mu_n}{\dd\gamma_n} \Big) \dd\gamma_n < \infty .
\end{equation*}
Then $\mu$ is absolutely continuous with respect to $\gamma$ and
\begin{equation*}
    \liminf_{n\to +\infty} \int_{\R^d} f\Big(\frac{\dd\mu_n}{\dd\gamma_n} \Big) \dd\gamma_n \geq \int_{\R^d} f\Big(\frac{\dd\mu}{\dd\gamma} \Big) \dd\gamma.
\end{equation*}
\end{proposition}
\subsection{Wasserstein distance and transport map}\label{wasser}

We recall that the Wasserstein distance in $\PP(\R^2)$,
is defined by
\begin{equation}\label{Kanto}
d_W^2(\mu,\nu):=\min_{\gamma\in\PP({\R^2\times\R^2})}\left\{\int_{\R^2\times\R^2}\!\!\!\!\!|x-y|^2\,d\gamma:
\,(\pi_1)_\#\gamma=\mu,\,(\pi_2)_\#\gamma=\nu
\right\}
\end{equation}
where $\pi_i$, $i=1,2$, denote the canonical projections on the factors. When $\mu$ is absolutely continuous with respect to the Lebesgue measure, the minimum problem~\eqref{Kanto}
has a unique solution $\gamma$ induced by a transport map $T_\mu^\nu$,
$\gamma=(\id,T_\mu^\nu)_\#\mu$. In particular, $T_\mu^\nu$ is the unique solution of the Monge optimal transport problem
$$
\min_{S:\R^2\longrightarrow\R^2}\left\{\int_{\R^2}|S-\id|^2\dd\mu:\ S_\#\mu=\nu\right\},
$$
of which~\eqref{Kanto} is the Kantorovich relaxed version. Finally, we recall that if also
$\nu$ is absolutely continuous with respect to Lebesgue measure, then
\begin{equation}\label{ot}
T_\nu^\mu\circ T_\mu^\nu=\id\quad\hbox{\rm $\mu$-a.e.}
\quad\hbox{\rm and}\quad
T_\mu^\nu\circ T_\nu^\mu=\id\quad\hbox{\rm $\nu$-a.e.}
\end{equation}
Since in this paper we deal only with absolutely continuous measures,
we identify the measures with their densities with respect to the Lebesgue measure.

\subsection{Boundedness from below of the functional $\EE$}
The following result is due to~\cite[Lemma~3.1]{CC08} and is a consequence of the Onofri inequality on the sphere~\cite{O82}:
\begin{lemma}[Onofri Inequality]
Let $H:\R^2 \to \R$ be defined by
\begin{equation*}
    H(x):=\frac{1}{\pi (1+|x|^2)^2}.
\end{equation*}
Then the following inequality holds:
\begin{equation}\label{OnofriIn}
    \int_{\R^2} e^\psi H \dd x \leq \exp \Big (\int_{\R^2} \psi H \dd x + \frac{1}{16\pi}\int_{\R^2} |\nabla\psi|^2 \dd x   \Big )
    \qquad \forall \psi \in \H^1(\R^2).
\end{equation}
\end{lemma}

\

We can now make use of this inequality to obtain the lower bound of the functional in \eqref{eq:functional}.

\begin{lemma}[Lower bound on $\EE$]\label{tintin}
Let $0<\chi < 8\pi$ and $\alpha> 0$. Then there exist constants $\nu>0$ and $C_1>0$ such that $\nu$ is independent of $\alpha$ and it holds
\begin{equation}\label{eq:38}
\begin{aligned}
    \EE[\rho,\phi] \ge & \,\frac{8\pi - \chi}{16 \pi\chi} \int_{\RR^2}  \rho |\log \rho| \dd x +\nu \left[\|\nabla \phi\|_{{\rm L}^2(\R^2)}^2 + \alpha \|\phi\|_{{\rm L}^2(\R^2)}^2 \right] \\
    & + \frac{3}{2\chi}\int_{\R^2}\rho\log H\dd x-C_1 \,,
    \qquad \forall \; (\rho,\phi)\in \PP({\R^2})\times \HH^1(\R^2).
\end{aligned}
\end{equation}
\end{lemma}

\begin{proof}
The proof follows the lines of \cite[Theorem~3.2]{CC08}.
Let $\delta \in (0,1)$ be a constant to be choosen later.
By Jensen's inequality, for all $\psi:\R^2\to\R$
\begin{equation*}
  \begin{aligned}
     \int_{\RR^2} \left[\frac{1-\delta}{\chi}\rho \log \rho -\psi\rho \right] \dd x 
&= - \frac{1-\delta}{\chi}\int_{\RR^2} \log \left(\frac{\ee^{\chi\psi/(1-\delta)}}{\rho} \right) \rho \dd x \nonumber\\
&\ge- \frac{1-\delta}{\chi}\log \left(\int_{\R^2} \ee^{\chi\psi/(1-\delta)} \dd x\right) \;.
  \end{aligned}
\end{equation*}
Applying this inequality to $\psi=\phi + (1-\delta)\log H/\chi$ we obtain
\begin{equation*}
  \begin{aligned}
   \int_{\RR^2}  \left[\frac{1-\delta}{\chi}\rho \log \rho -\phi\rho \right]\dd x = &\, \int_{\RR^2}  \left[\frac{1-\delta}{\chi}\rho \log \rho -\psi\rho \right]\dd x + \int_{\R^2}\rho(\psi-\phi)\dd x\\
\ge&\,- \frac{1-\delta}{\chi}\log \left(\int_{\R^2} H\ee^{\chi\phi/(1-\delta)} \dd x\right)\\
&+\frac{1-\delta}{\chi}\int_{\R^2}\rho\log H\dd x .
  \end{aligned}
\end{equation*}
By Onofri's inequality \eqref{OnofriIn} we obtain, for any $\eps>0$,
\begin{equation*}
  \begin{aligned}
   \int_{\RR^2}  \left[\frac{1-\delta}{\chi}\rho \log \rho -\phi\rho \right]\dd x \ge &\,- \frac{1-\delta}{\chi}\left(\int_{\R^2} \frac{\chi\,\phi}{1-\delta}H\dd x + \frac{\chi^2 \|\nabla \phi\|_{{\rm L}^2(\R^2)}^2}{16\,\pi (1-\delta)^2}\right)\\
& +\frac{1-\delta}{\chi}\int_{\R^2}\rho\log H\dd x\\
\ge&\, - \int_{\R^2} \phi\,H\dd x - \frac{\chi \|\nabla \phi\|_{{\rm L}^2(\R^2)}^2}{16\,\pi (1-\delta)}\\
&+\frac{1-\delta}{\chi}\int_{\R^2}\rho\log H\dd x\\
\ge&\, - \frac{\eps}{2}\|\phi\|_{{\rm L}^2(\R^2)}^2 - \frac{1}{2\eps}\|H\|_{{\rm L}^2(\R^2)}^2 - \frac{\chi \|\nabla \phi\|_{{\rm L}^2(\R^2)}^2}{16\,\pi (1-\delta)}\\
&+\frac{1-\delta}{\chi}\int_{\R^2}\rho\log H\dd x .
  \end{aligned}
\end{equation*}
Choosing $\eps=\alpha\chi/(8\,\pi(1-\delta))>0$, we obtain
\begin{multline*}
  \EE[\rho,\phi] \ge \frac\delta\chi \int_{\RR^2}  \rho \log \rho \dd x - \frac1{2\eps}\|H\|_{{\rm L}^2(\R^2)}^2 + \nu \left[\|\nabla \phi\|_{{\rm L}^2(\R^2)}^2 + \alpha \|\phi\|_{{\rm L}^2(\R^2)}^2 \right] \\+ \frac{1-\delta}{\chi}\int_{\R^2}\rho\log H\dd x 
\end{multline*}
with $\nu:=1/2-{\chi}/{(16\,\pi (1-\delta))}$.
By Carleman's estimate~\eqref{eq:carleman}
\begin{equation*}
\begin{aligned}
    \EE[\rho,\phi] \ge &\,\frac\delta\chi \int_{\RR^2}  \rho |\log \rho| \dd x + \frac{2\delta}\chi\int_{\RR^2}  \rho\log H\dd x -\frac{2\delta}{\ee\chi} - \frac{4\pi(1-\delta)}{\alpha\chi}\|H\|_{{\rm L}^2(\R^2)}^2  \\
&+ \nu\left[\|\nabla \phi\|_{{\rm L}^2(\R^2)}^2 + \alpha \|\phi\|_{{\rm L}^2(\R^2)}^2 \right]+ \frac{1-\delta}{\chi}\int_{\R^2}\rho\log H\dd x\\
= &\, \frac\delta\chi \int_{\RR^2}  \rho |\log \rho| \dd x +\nu \left[\|\nabla \phi\|_{{\rm L}^2(\R^2)}^2 + \alpha \|\phi\|_{{\rm L}^2(\R^2)}^2 \right] + \frac{1+\delta}{\chi}\int_{\R^2}\rho\log H\dd x\\
&-\frac{2\delta}{\ee\chi} - \frac{4\pi(1-\delta)}{\alpha\chi}\|H\|_{{\rm L}^2(\R^2)}^2 .
\end{aligned}
\end{equation*}
Since $\chi<8\pi$ we can take $\delta=(8\pi - \chi)/(16\pi)$.
Observing that $\delta<1/2$, \eqref{eq:38} follows with $\nu=(8\pi -\chi)/(16\pi +2\chi)>0$
and $C_1={8\pi -\chi}/{(8\pi\ee\chi)} + {8\pi+\chi}/({4\alpha\chi})\,\|H\|_2^2$.
\end{proof}



\section{One Step Variational Problem}\label{mini}

\subsection{Existence of minimizers}
\begin{proposition}[Existence of minimizers]\label{prop:lsc}
If $0<\chi < 8\pi$ and $\alpha>0$ then for any $(\bar \rho, \bar \phi)\in \KK\times \HH^1(\RR^2)$ the functional
  \begin{equation*}
   \FF[\rho,\phi] := \frac{1}{2h}\ \left[ \frac{d_W^2(\rho,\bar \rho)}{\chi} + \tau\ \|\phi-\bar \phi\|_{{\rm L}^2(\R^2)}^2 \right] + \EE[\rho,\phi]\,
\end{equation*}
is bounded from below in $\PP({\R^2})\times \HH^1(\R^2)$ and sequentially lower semi-continuous with respect to the narrow topology in $\PP({\R^2})$ and the weak topology in $\LL^2(\R^2)$ for all $h,\tau>0$. Moreover the sub-levels of $\FF$ are sequentially compact with respect to those same topologies. In particular, the functional $\FF$ admits a minimizer in $\PP({\R^2})\times \HH^1(\R^2)$.
\end{proposition}
\begin{proof}
$\bullet$ Lower bound for $\FF$: By the Young Inequality and the definition of the distance $d_W$, 
we have that
\begin{equation*}
    d_W^2(\rho,\bar\rho)\geq \frac{1}{2}\int_{\RR^2} |x|^2 \rho(x)\dd x - \int_{\RR^2}|x|^2 \bar\rho(x)\dd x .
\end{equation*}
Since $ \log H(x)= - \log \pi - 2\log(1+|x|^2) $, we deduce
$$
\begin{aligned}
    \frac{3}{2\chi}\int_{\RR^2} \rho\log H \dd x + \frac{1}{2h}d_W^2(\rho,\bar\rho)  \geq &\,
       \int_{\RR^2} \Big(\frac{1}{4h}|x|^2 -  \frac{3}{\chi}\log(1+|x|^2)\Big)\rho(x)\dd x \\
      &-\frac{3}{2\chi}\log \pi- \frac{1}{2h} \int_{\RR^2} |x|^2 \bar\rho(x)\dd x\;.
\end{aligned}
$$
This quantity is bounded from below because the function $s \in [0,\infty) \mapsto s/(4h) -  \frac{3}{\chi}\log(1+s)$ is bounded from below. Using Lemma \ref{tintin}, we have thus obtained that there exist $C_2=\nu\min\{1,\alpha\}>0$, $C_3=C_3(h) \in \R$ such that for all $(\rho,\phi) \in \PP_2({\R^2}) \times \HH^1(\R^2)$, we get
\begin{equation}\label{b}
    \FF[\rho,\phi] \ge \frac{8\pi - \chi}{16 \pi \chi} \int_{\R^2}\rho |\log \rho| + C_2 \norm{\phi}^2_{\HH^1(\R^2)} + C_3 \; .
\end{equation}

\noindent$\bullet$ Lower semi-continuity of $\FF$: we take a sequence $(\rho_n)_{n \ge 1}\in\PP_2({\R^2})$ narrowly convergent to $\rho$ and $\phi_n\in \HH^1(\R^2)$ such that $(\phi_n)_{n \ge 1}$ weakly converges to $\phi$ with respect to the $\LL^2(\R^2)$-topology. Denoting by
$$\gamma_n = \ee^{\chi \phi_n}H \dd x, \quad \gamma = \ee^{\chi\phi}H \dd x, \quad \mu_n=\rho_n\dd x, \quad \mu=\rho\dd x$$
the functional $\FF$ can be rewritten in the following form
\begin{eqnarray}
    \FF[\rho_n,\phi_n] & = &\frac1\chi\int_{\R^2} \left[\frac{d\mu_n}{\dd\gamma_n} \log\left( \frac{\dd\mu_n}{\dd\gamma_n}\right) +\frac{1}{\ee}\right] \dd\gamma_n - \frac{1}{\ee\chi}\int_{\R^2} \ee^{\chi \phi_n}H\dd x\label{lsc1}\\
    &&\quad + \frac1\chi\int_{\R^2}\rho_n\log H \dd x + \frac{1}{2} \left( \|\nabla \phi_n\|_{{\rm L}^2(\R^2)}^2 + \alpha \|\phi_n\|_{{\rm L}^2(\R^2)}^2\right) \label{lsc3}\\
    &&\quad +\frac{1}{2h}\left(\frac{d_W^2(\rho_n,\bar\rho)}\chi+\tau\norm{\phi_n-\bar \phi}^2_{{\rm L}^2(\R^2)} \right) \label{lsc4}.
\end{eqnarray}
$*$ Lower semi-continuity of~\eqref{lsc1}: By \eqref{b} the sequence $(\phi_n)_{n \ge 1}$ is bounded in $\HH^1(\R^2)$, and after possibly extracting a sub-sequence, we may assume that
\begin{equation}\label{16b}
(\phi_n(x))_{n \ge 1} \to \phi(x) \mbox{ for a.e. }x\in \R^2.
\end{equation}
We prove that $(\gamma_n)_{n \ge 1}$ narrowly converges to $\gamma$, {\it i.e.}  for all $\varphi \in \C_b(\R^2)$
\begin{equation}\label{eq:nc}
\lim_{n \to \infty}\int_{\R^2} \ee^{\chi\phi_n(x)}H(x)\varphi(x) \dd x = \int_{\R^2} \ee^{\chi\phi(x)}H(x)\varphi(x) \dd x .
\end{equation}
Indeed, $H \!\dd x$ is a finite measure in $\R^2$ and, by Onofri's inequality~\eqref{OnofriIn} and the boundedness of the $(\phi_n)_{n \ge 1}$ in $\HH^1(\R^2)$, we deduce
\begin{equation*}
  \begin{aligned}
    \int_{\R^2} \left(\ee^{\chi\phi_n} \varphi\right)^2 H \dd x &\le \norm{\varphi}^2_{{\rm L}^\infty(\R^2)} \int_{\R^2} \ee^{2\chi\phi_n}H\dd x \\
&\le \norm{\varphi}^2_{{\rm L}^\infty(\R^2)} \exp\left[2\chi\int_{\R^2} \phi_n H \dd x + \frac{\chi^2}{4\pi}\|\nabla \phi_n\|_{{\rm L}^2(\R^2)}^2 \right]\\
&\le \norm{\varphi}^2_{{\rm L}^\infty(\R^2)} \exp\left[2\chi\|\phi_n\|_{{\rm L}^2(\R^2)}\|H\|_{{\rm L}^2(\R^2)}  + \frac{\chi^2}{4\pi}\|\nabla \phi_n\|_{{\rm L}^2(\R^2)}^2 \right] \\ & \leq C.
  \end{aligned}
\end{equation*}
Therefore, $(\ee^{\chi\phi_n} \varphi)_{n \ge 1}$ is bounded in $\LL^2(\RR^2;H\!\dd x)$, and thus it is uniformly integrable with respect to the measure  $H\!\dd x$. Recalling~\eqref{16b} we may apply Vitali's Dominated Convergence Theorem and obtain~\eqref{eq:nc}. By Proposition \ref{prop:lscIF} applied to the non-negative convex function $f(s)=s\log s +1/\ee$ we obtain the lower semi-continuity of the right hand side of \eqref{lsc1}.

\noindent$*$ Lower semi-continuity of~\eqref{lsc3}: Since the lower semi-continuity property of $\FF$ is obvious if $\limsup_{n \to \infty}\FF[\rho_n,\phi_n] = \infty$, it is not restrictive to assume that there exists a constant $C$ such that $\FF[\rho_n,\phi_n] \leq C$. Combining the upper bound $\FF[\rho_n,\phi_n] \leq C$ with the lower bound~\eqref{eq:38} on $\EE$ we deduce that $(d_W^2(\rho_n,\bar \rho))_{n \ge 1}$ is bounded so that $(\rho_n)_{n \ge 1}$ is bounded in $\PP_2(\R^2)$. Since $\log H=o(|x|^2)$ as $|x| \to \infty$, this last bound and the narrow convergence imply the convergence
$$\lim_{n \to \infty}\int_{\R^2}\rho_n\log H \dd x = \int_{\R^2} \rho\log H\dd x\ .$$

By \eqref{b} the sequence $(\phi_n)_{n \ge 1}$ is bounded in $\HH^1(\R^2)$, and after possibly extracting a sub-sequence, we may assume that it converges weakly to $\phi$ in $\HH^1(\R^2)$. Thus, the lower semicontinuity of the last two terms in \eqref{lsc3} are obvious.

\noindent$*$ Lower semi-continuity of~\eqref{lsc4}: it follows from the lower semi-continuity of the Wasserstein distance and the lower semi-continuity of the $\LL^2$ norm.
\end{proof}


\subsection{Improved regularity of the minimizers}

\subsubsection{Matthes-McCann-Savar\'e flow interchange technique}\label{sec:MMS}
We will use a variation of a powerful method developed by Matthes-McCann-Savar\'e in~\cite{MMS09}.

We denote by $X$ the metric space $\PP(\R^2)\times \L^2(\R^2)$ endowed with the metric
\begin{equation}\label{def:d}
d^2(u_1,u_2)= \frac{1}{\chi}d_W^2(\rho_1,\rho_2) + \tau\ \|\phi_1-\phi_2\|_{{\rm L}^2(\R^2)}^2,
\end{equation}
where $u_i=(\rho_i,\phi_i)$, $i=1,2$.

The scheme \eqref{scheme:jko} can be rephrased as
\begin{equation}\label{num22}
  \mbox{$(\rho_h^n,\phi_h^n)=u_{h}^{n}$ minimises in $X$ the functional} \quad u \mapsto \frac{1}{2 h}d^2(u,u_h^{n-1})+\EE[u]\;,
\end{equation}
for all $n \ge 1$ and $h>0$ starting from $u_h^0=(\rho_0,\phi_0)$.

Assume that $\V:X\to (-\infty,+\infty]$ is a proper lower semi continuous functional
that admits a continuous semigroup $(S^{\V}_t)_{t\ge 0}$ in $Dom(\V)$
satisfying the following \emph{Evolution Variational Inequality} (EVI)
  \begin{equation}\label{EVIgen}
   \frac12\frac{d^2(S^{\V}_t(u),\bar u)-d^2(u,\bar u)}t
    + \V[S^{\V}_t(u)] \leq \V[\bar u] , \qquad  u,\bar u\in Dom(\V), \;  t>0.
  \end{equation}
The dissipation of $\EE$ along the
 flow $(S^{\V}_t)_{t\ge 0}$ associated to $\V$ is defined by
\begin{equation*}
 D^\V \EE[u]:= \limsup_{t \downarrow 0} \frac{\EE[u]-\EE[S_t^{\V}(u)]}{t}\;.
\end{equation*}
By the minimising scheme \eqref{num22}, for any $u \in Dom(\V)$ we have
\begin{equation*}
  \frac{1}{2h}d^2(u_h^n,u_h^{n-1})+\EE[u_h^{n}] \le \frac{1}{2h}d^2(u,u_h^{n-1})+\EE[u]\;.
\end{equation*}
Choosing $u=S^\V_t u_h^{n}$ for $t >0$, and dividing by $t$ we obtain
\begin{equation*}
 \frac{\EE[u_h^n]-\EE[S^\V_t (u_h^n)]}{t} \le \frac{1}{2h} \left[\frac{d^2(S^\V_t (u_h^n),u_h^{n-1})- d^2(u_h^n,u_h^{n-1})}{t} \right]\;.
\end{equation*}
As $\V$ satisfies \eqref{EVIgen} we have
\begin{equation}\label{v 1}
 \frac{\EE[u_h^n]-\EE[S^\V_t (u_h^n)]}{t}  \le \frac{\V[u_h^{n-1}]-\V[S^\V_t (u_h^n)]}{h}\;.
\end{equation}
Since $\V$ is lower semi continuous, passing to the limit $t\to 0$ we obtain
\begin{equation}\label{mms09}
D^\V \EE[u_h^n]\le   \frac{\V[u_h^{n-1}] -\V[u_h^n]}{h} \;.
\end{equation}
So that the differential estimate~\eqref{v 1} is converted into the discrete estimate~\eqref{mms09} for the approximation scheme~\eqref{num22},
which could provide additional information on $u_h^n$ according to the properties of $D^\V \EE$.
In particular when $D^\V \EE \ge 0$ we can expect to control $u^n_h$ by the prior state $u^{n-1}_h$ if $\V[u^{n-1}_h]<+\infty$, which is the situation dealt with  in~\cite{MMS09}.
In our case we do not have the  nice property $D^\V \EE \ge 0$  but
we can decompose $D^\V \EE$ into a positive contribution and a controlled remainder (see Lemma~\ref{impro}).

\subsubsection{Regularity of minimizers}
As already mentioned in the introduction we turn to additional regularity properties of $\rho$ and $\phi$ using the method introduced above.
We define the functional $\V:\PP_2(\R^2)\times \L^2(\R^2)\to(-\infty,+\infty]$
by
\begin{equation*}
\V[\rho,\phi]= \frac{1}{\chi}\int_{\RR^2} \rho(x)\log\rho(x)\dd x + \frac{\tau}{2}\int_{\RR^2}\left[|\nabla \phi(x)|^2+\alpha\,\phi(x)^2\right]\dd x \;,
\end{equation*}
if $(\rho,\phi)\in\KK\times \H^1(\R^2)$ and $\V[\rho,\phi]= +\infty$ otherwise.

It is well known that $\V$ is lower semi continuous with respect to the narrow topology in $\rho$ and the $\L^2$ weak topology in $\phi$.
Moreover, $\V$ generates a continuous semigroup in $Dom(\V)=\KK\times \H^1(\R^2)$ satisfying the EVI \eqref{EVIgen} (see \cite{AGS08}).

\begin{lemma}[Improved regularity of minimizers]\label{impro}
 Consider the sequence of minimizers $(\rho_h^n,\phi_h^n)$ and
 let $\Lambda$ satisfy $$\Lambda \ge \int_{\R^2} \rho_h^n\log \rho_h^n\dd x +
 4\int_{\R^2} \rho_h^n\log(1+|x|^2)\dd x +\frac{2}{\ee}+2\log\pi+16.$$
Then  $\rho_h^n \in W^{1,1}(\RR^2)$, $\nabla\rho_h^n/\rho_h^n \in \L^2(\rho_h^n)$, $\phi_h^n \in \HH^2(\RR^2)$, and there exists a constant
$C(\Lambda)>0$ such that
\begin{align}\label{impreg}
 \frac{1}{2\chi} \int_{\R^2}\left|\frac{\nabla\rho_h^n}{\rho_h^n}\right|^2\rho_h^n\dd x + \frac{1}{2}\|\Delta \phi_h^n + \rho_h^n - \alpha \phi_h^n\|_2^2
 \le&\, \frac{1}{h}\left(\V(\rho_h^{n-1},\phi_h^{n-1}) - \V(\rho_h^{n},\phi_h^{n})\right) \nonumber\\
 & + C(\Lambda) + \frac{\alpha}{2} \|\phi_h^n\|_2^2\;.
\end{align}
\end{lemma}

\begin{proof}
We use the notation $u_h^n=(\rho_h^n,\phi_h^n)$ and
 $u(t)=(\sigma(t),\Phi(t))=S^\V_t(\rho_h^n,\phi_h^n)=S^\V_t(u_h^n)$ for $t\ge 0$.
 The functions $\sigma$, $\Phi$ solve the equations
\begin{equation}\label{eq:i1}
    \begin{aligned}
   \partial_t \sigma &= \Delta \sigma  &\mbox{in } (0,\infty) \times \RR^2,\quad &\sigma(0)=\rho_h^n, \\
   \partial_t \Phi &= \Delta \Phi - \alpha \Phi &\mbox{in } (0,\infty) \times \RR^2, \quad &\Phi(0)=\phi_h^n,
    \end{aligned}
  \end{equation}
and satisfy the following identities:
\begin{equation*}
  \frac{\dd}{\dd t} \int_{\RR^2} \sigma(t)\log\sigma(t)\dd x = -  \int_{\R^2}\left|\frac{\nabla\sigma(t)}{\sigma(t)}\right|^2\sigma(t)\dd x
   >-\infty, \qquad\forall t>0,
\end{equation*}
\begin{equation*}
 \frac{1}{2} \frac{\dd}{\dd t} \int_{\RR^2}|\nabla \Phi(x)|^2+\alpha\,\Phi(x)^2\dd x  =
  - \int_{\RR^2} |\Delta \Phi(t)  - \alpha \Phi(t)|^2 \dd x >-\infty, \qquad \forall t>0.
\end{equation*}

\noindent{\bf Step 1 -} We give an estimate of $\|\sigma(t)\|_2^2$.
From the inequality \eqref{BHN}
\begin{equation*}
  \|\sigma(t)\|_2^2 \le {\eps}  \left\|\frac{\nabla\sigma(t)}{\sigma(t)}\right\|_{\L^2(\sigma(t))}^2\|\sigma(t)\log\sigma(t)\|_{{\rm L}^1(\R^2)} + L_{\eps}.
\end{equation*}
From Carleman's estimate~\eqref{eq:carleman}, we deduce
\begin{equation*}
  \begin{aligned}
    \|\sigma(t) \log \sigma(t) \|_1 &\le \int_{\R^2}\sigma(t)\log\sigma(t)\dd x +\frac2\ee +2\log\pi + 4 \int_{\R^2} \sigma(t)\log(1+|x|^2)\dd x.
  \end{aligned}
\end{equation*}
Since $ t\mapsto \int_{\R^2}\sigma(t)\log\sigma(t)\dd x$ is decreasing in $[0,+\infty)$ and
\begin{align*}
  \frac{\dd}{\dd t}\int_{\R^2}\sigma(t)\log(1+|x|^2)\dd x &= \int_{\R^2}\sigma(t)\Delta(\log(1+|x|^2))\dd x\\ 
  &=  \int_{\R^2} \sigma(t) \frac{4}{(1 + |x|^2)^2} \dd x \le 4\;,
\end{align*}
we infer that
\begin{align*}
    \int_{\R^2}  \sigma(t)|\log \sigma(t)| \dd x &\le \int_{\R^2}  \sigma(t)\log \sigma(t) \dd x +\frac2\ee +2\log\pi + 4 \int_{\R^2} \sigma(t)\log(1+|x|^2)\dd x\\
& \le \int_{\R^2}\rho_h^k\log\rho_h^k \dd x +\frac2\ee + +2\log\pi + 4 \int_{\R^2} \rho_h^k\log(1+|x|^2)\dd x + 16t \\
&\le \Lambda\;
\end{align*}
for $t\in (0,1]$.
We thus obtain that
\begin{equation}\label{eq:i4b}
  \begin{aligned}
  \|\sigma(t)\|_{{\rm L}^2(\R^2)}^2 & \le {\eps} \Lambda \left\|\frac{\nabla\sigma(t)}{\sigma(t)}\right\|_{\L^2(\sigma(t))}^2+ L_{\eps}.
  \end{aligned}
\end{equation}

\noindent{\bf Step 2 -} Instead of computing $D^\V \EE[u_h^n]$
we use the regularity properties for the solutions of the equations \eqref{eq:i1} and we compute $D^\V\EE[u(t)]$ for $t>0$.
In this case we claim that
\begin{equation}\label{dissut}
 D^\V\EE[u(t)] = \DD[u(t)]-\Re[u(t)],\qquad t>0
\end{equation}
where
$$
\DD[u(t)]:=\frac{1}{\chi} \int_{\R^2}\left|\frac{\nabla\sigma(t)}{\sigma(t)}\right|^2\sigma(t)\dd x
+ \|\Delta \Phi(t) + \sigma(t) - \alpha \Phi(t)\|_2^2$$
and
$$\Re[u(t)]:= \|\sigma(t)\|_{{\rm L}^2(\R^2)}^2 -\alpha \int_{\R^2}\sigma(t)\Phi(t)\dd x.$$
Indeed, owing to the smoothness of the solutions of \eqref{eq:i1}, we have that for $t>0$
\begin{eqnarray*}
-D^\V\EE[u(t)]=  \frac{\dd}{\dd t}\EE[\sigma,\phi]
  &=&-\frac{1}{\chi} \int_{\RR^2} \frac{|\nabla \sigma|^2}{\sigma} \dd x -\int_{\RR^2} \Phi \,\partial_t \sigma\dd x\\
&&\quad -\int_{\RR^2} \sigma \,\partial_t \Phi\dd x - \|\Delta \Phi - \alpha \Phi\|^2_{{\rm L}^2(\R^2)}\\
&=&-\frac{1}{\chi} \int_{\RR^2} \frac{|\nabla \sigma|^2}{\sigma} \dd x -\int_{\RR^2} \Phi\, \Delta \sigma\dd x\\
&&\quad  -\int_{\RR^2} \sigma (\Delta \Phi - \alpha \Phi)\dd x- \|\Delta \Phi - \alpha \Phi\|^2_{{\rm L}^2(\R^2)}\\
&=&-\frac{1}{\chi} \int_{\RR^2} \frac{|\nabla \sigma|^2}{\sigma} \dd x -2\int_{\RR^2} \sigma (\Delta \Phi  - \alpha \Phi)\dd x-\alpha\int_{\RR^2} \sigma\, \Phi\dd x\\
&&\quad  -\|\sigma\|^2_{{\rm L}^2(\R^2)}+\|\sigma\|^2_{{\rm L}^2(\R^2)} - \|\Delta \Phi - \alpha \Phi\|^2_{{\rm L}^2(\R^2)}\\
&=&-\frac{1}{\chi} \int_{\RR^2} \frac{|\nabla \sigma|^2}{\rho} \dd x- \|\Delta \Phi +\sigma - \alpha \Phi\|^2_{{\rm L}^2(\R^2)} + \|\sigma\|^2_{{\rm L}^2(\R^2)} \\
&&\quad -\alpha\int_{\RR^2} \sigma\,\Phi\dd x \;.
\end{eqnarray*}
hence~\eqref{dissut}.

Taking into account that $t\mapsto\|\Phi(t)\|_2^2$ is decreasing in $[0,\infty)$,
we deduce from~\eqref{eq:i4b} and the definition of $\DD(u(t))$, the following estimate for $\Re[u(t)]$
\begin{equation*}
\begin{aligned}
  \Re[u(t)] &\le \left(1+\frac{\alpha}{2}\right)\|\sigma(t)\|_{{\rm L}^2(\R^2)}^2 + \frac{\alpha}{2}\|\Phi(t)\|_{{\rm L}^2(\R^2)}^2 \\
  &\le  {\eps}\frac{\alpha+2}{2}\Lambda \left\|\frac{\nabla\sigma(t)}{\sigma(t)}\right\|_{\L^2(\sigma(t))}^2+ L_{\eps} + \frac{\alpha}{2}\|\Phi_h^n\|_{{\rm L}^2(\R^2)}^2 \\
  &\le {\eps}\frac{\alpha+2}{2}\Lambda\chi\DD[u(t)]+ L_{\eps} + \frac{\alpha}{2}\|\phi_h^n\|_{{\rm L}^2(\R^2)}^2.
  \end{aligned}
\end{equation*}
Choosing $\eps=1/(\chi \Lambda (\alpha +2))$ we obtain
\begin{equation*}
  \Re[u(t)] \le \frac{1}{2}\DD[u(t)]+ C(\Lambda) + \frac{\alpha}{2}\|\phi_h^n\|_{{\rm L}^2(\R^2)}^2,
\end{equation*}
where $C(\Lambda)=L_\eps$ with the choice of $\eps$ above.
Then we have
\begin{equation}\label{di}
 \DD[u(t)] = D^\V\EE[u(t)] + \Re[u(t)] \le D^\V\EE[u(t)] + \frac{1}{2}\DD[u(t)]+ C(\Lambda) + \frac{\alpha}{2}\|\phi_h^n\|_{{\rm L}^2(\R^2)}^2.
\end{equation}

\noindent{\bf Step 3 -} The function $t\mapsto \EE[u(t)]$ is continuous in $[0,+\infty)$.
This property is clear for $t>0$ owing to the smoothness of $u(t)$. We only have to prove the continuity at $0$.
Recalling that as $t\to 0$
\begin{equation}\label{contE}
\begin{aligned}
&\int_{\RR^2} \sigma(t)\log\sigma(t)\dd x \to \int_{\RR^2} \rho_h^n\log\rho_h^n\dd x,\\
& \frac{1}{2}\int_{\RR^2}\left[|\nabla \Phi(t)|^2+\alpha\,\Phi(t)^2\right]\dd x\to \frac{1}{2}\int_{\RR^2}\left[|\nabla \phi_h^n|^2+\alpha\,|\phi_h^n|^2\right]\dd x,
\end{aligned}
\end{equation}
we have to prove that $\int_{\RR^2}\Phi(t)\sigma(t) \dd x\to\int_{\RR^2}\phi_h^n\,\rho_h^n \dd x $.

Introducing $A(s)=(s+1)\log(s+1)-s$ and its convex conjugate $A^*(s)=e^s-s-1$ we recall Young's inequality $s\,s' \le A(s)+A^*(s')$ for $s,s'\in [0,\infty)$. We also recall a variant of Moser-Trudinger's inequality, see~\cite{masmoudi}:
\begin{equation}\label{masmoudi}
  \int_{\R^2}\left(\ee^{2\pi u^2}-1\right)\dd x\le C \|u\|^2_{{\rm L}^2(\R^2)} \quad\mbox{for $u\in \H^1(\R^2)$ such that $\|\nabla u\|_{{\rm L}^2(\R^2)}\le 1$}\;.
\end{equation}

Let $\varepsilon \in (0,1)$ be such that 
\begin{equation}
  \label{eq:bo}
\varepsilon \sup_{t \in [0,1]} \|\Phi(t)\|_{\H^1(\R^2)} \le 1\;.
\end{equation}
Since $\Phi(t)$ converges to $\phi_h^n$ in $\H^1(\R^2)$ as $t \to 0$, there is $t_\eps \in (0,1)$ such that
\begin{equation}
  \label{eq:boum}
  \|\Phi(t)-\phi_h^n\|_{\H^1(\R^2)} \le 1 \quad \mbox{for $t \in [0,t_\eps]$.}
\end{equation}
Let $t \in [0,t_\eps]$. On the one hand, it follows from Young's inequality,~\eqref{masmoudi}, and~\eqref{eq:bo} that
\begin{align}\label{eq:bam}
  \int_{\R^2} |\Phi(t)||\sigma(t)-\rho_h^n|\dd x &\le \int_{\R^2} A\left( \frac{|\sigma(t)-\rho_h^n|}{\eps}\right)\dd x+\int_{\R^2} A^*(\eps |\Phi(t)|)\nonumber\\
&\le \int_{\R^2} A\left( \frac{|\sigma(t)-\rho_h^n|}{\eps}\right)\dd x+C \eps^2 \sup_{s \in [0,1]}\|\Phi(t)\|^2_{{\rm L}^2(\R^2)}\;.
\end{align}
On the other hand Young's inequality,~\eqref{masmoudi}, and~\eqref{eq:boum} give
\begin{align}\label{eq:bim}
  \int_{\R^2} |\rho_h^n||\Phi(t)-\phi_h^n|\dd x &\le \int_{\R^2} A(\eps\,\rho_h^n)+ \int_{\R^2} A^*\left( \frac{|\Phi(t)-\phi_h^n|}{\eps}\right)\dd x\nonumber\\
&\le \int_{\R^2} A(\eps\,\rho_h^n) \dd x+\frac{C}{\eps^2}\|\Phi(t)-\phi_h^n\|^2_{{\rm L}^2(\R^2)}\;.
\end{align}
Since $\sigma(t)\log\sigma (t) \to \rho_h^n \log \rho_h^n$ in $\L^1(\R^2)$ and $\Phi(t) \to \phi^n_h$ in $\L^2(\R^2)$, we let $t \to 0$ in~\eqref{eq:bam} and~\eqref{eq:bim} and obtain
\begin{equation*}
  \limsup_{t \to 0} \left| \int_{\R^2} \left(\Phi(t)\,\sigma(t) -\phi_h^n\, \rho_h^n \right)\dd x\right| \le C \eps^2 \sup_{s \in [0,1]}\|\Phi(t)\|^2_{{\rm L}^2(\R^2)}+\int_{\R^2} A(\eps\,\rho_h^n)\dd x\;.
 \end{equation*}
We finally use the integrability of $\rho_h^n \log \rho_h^n$ to pass to the limit as $\eps \to 0$ in the above inequality and conclude that
\begin{equation*}
  \lim_{t \to 0} \left| \int_{\R^2} \left(\Phi(t)\,\sigma(t) -\phi_h^n \,\rho_h^n \right)\dd x\right|=0\,,
\end{equation*}
thereby completing the proof of the continuity of $t \mapsto \EE[u(t)]$.

\noindent{\bf Step 4 -} By the Lagrange theorem, since $t\mapsto \EE[u(t)]$  is continuous at $t=0$ and differentiable at $t>0$,
for every $t>0$ there exists $\theta(t)\in(0,t)$ such that
\begin{equation*}
  \frac{\EE[u_h^n] - \EE[u(t)]}{t} = D^\V\EE[u(\theta(t))].
\end{equation*}
From \eqref{v 1}, we obtain $D^\V\EE[u(\theta(t))] \le \frac{1}{h}\left(\V(u_h^{n-1}) - \V(S^\V_t(u_h^n))\right)$, and finally by \eqref{di}
\begin{equation}\label{eq:i6}
  \frac12\DD[u(\theta(t))] \le \frac{1}{h}\left(\V(u^{n-1}) - \V(S^\V_t(u_h^n))\right) + C(\Lambda) +\frac{\alpha}{2}\|\phi_h^n\|_{{\rm L}^2(\R^2)}^2.
\end{equation}
Then $\limsup_{t\to 0} \DD[u(\theta(t))] <+\infty$ due to \eqref{contE}.

Denoting by $\sigma_k=\sigma(\theta(t_k))$ and $\Phi_k=\Phi(\theta(t_k))$ sequences given by $t_k\to 0$ as $k\to+\infty$, we have
\begin{equation}\label{la}
    \limsup_{k\to+\infty} \int_{\R^2}\left|\frac{\nabla\sigma_k}{\sigma_k}\right|^2\sigma_k\dd x <+\infty
\end{equation}
and
\begin{equation*}
    \limsup_{k\to+\infty} \|\Delta\Phi_k + \sigma_k - \alpha\Phi_k\|_{{\rm L}^2(\R^2)}^2  <+\infty.
\end{equation*}
Moreover, by \eqref{eq:i4b} and \eqref{la}  we obtain
\begin{equation*}
\limsup_{k\to+\infty}\|\sigma_k\|_{{\rm L}^2(\R^2)}^2<+\infty.
\end{equation*}
By weak compactness in $\L^2(\R^2)$, taking into account that $\sigma_k\to \rho_h^n$ narrowly and $\Phi_k\to \phi_h^n$ strongly in $\HH^1(\R^2)$,
we pass to the limit by lower semicontinuity and we obtain that
$\rho_h^n\in\L^2(\R^2)$, $\Delta\phi_h^n\in\L^2(\R^2)$, $\phi_h^n\in\HH^2(\R^2)$ and
\begin{equation}\label{ee}
    \|\Delta\phi_h^n+ \rho_h^n - \alpha\phi_h^n\|_{{\rm L}^2(\R^2)}^2 \le \liminf_{k\to+\infty} \|\Delta\Phi_k + \sigma_k - \alpha\Phi_k\|_{{\rm L}^2(\R^2)}^2.
\end{equation}
Finally, by Proposition \ref{prop:vf}, defining
$v_k:={\nabla\sigma_k}/{\sigma_k}$
there exists $v\in \L^2(\R^2,\sigma_h^n;\R^2)$ such that, up to a subsequence,
\begin{equation}\label{a}
    \int_{\R^2} \varphi\cdot v_k \sigma_k \dd x \dd t \to \int_{\R^2} \varphi\cdot v \rho_h^n \dd x \dd t,
\end{equation}
for every $\varphi\in \C^\infty_0(\R^2,\R^2)$.
Since $v \rho_h^n \in \L^1(\R^2)$ and
\begin{equation*}
\int_{\R^2} \varphi\cdot v_k \sigma_k \dd x  =
     \int_{\R^2} \varphi\cdot \nabla\sigma_k \dd x =
     - \int_{\R^2} (\nabla\cdot\varphi) \sigma_k \dd x
     \to - \int_{\R^2} (\nabla\cdot\varphi) \rho_h^n \dd x ,
\end{equation*}
we deduce from \eqref{a} that, $v \rho_h^n = \nabla\rho_h^n$ and $\rho_h^n\in W^{1,1}(\R^2)$.
Finally, the lower semicontinuity property \eqref{VFlsc} yields that
\begin{equation}\label{eee}
    \int_{\R^2}\left|\frac{\nabla\rho_h^n}{\rho_h^n}\right|^2\rho_h^n\dd x
    \le \liminf_{k\to+\infty} \int_{\R^2}\left|\frac{\nabla\sigma_k}{\sigma_k}\right|^2\sigma_k\dd x.
\end{equation}
The final inequality \eqref{impreg} follows from \eqref{eq:i6}, \eqref{eee}, \eqref{ee}, \eqref{contE}, and the definition of the dissipation $\DD$.
\end{proof}

\subsection{The Euler-Lagrange equation}
\begin{lemma}[Euler-Lagrange equation]\label{lem:b7}
Let $0<\chi < 8\pi$, $(\rho_0,\phi_0)\in\KK \times \HH^1(\R^2)$ and $h>0$.
If $(\rho_h^n,\phi_h^n)$ is the sequence of the scheme \eqref{scheme:jko}, then
\begin{equation}\label{ELeq1}
 \int_{\RR^2} \zeta\cdot \left( \nabla\rho_h^n - \chi\rho_h^n\nabla \phi_h^n \right) \dd x
 = \frac1{{h}}\int_{\RR^2}\left[(T_{\rho_h^n}^{\rho_h^{n-1}}-\id)\cdot \zeta\right] \rho_h^n\dd x\,
\end{equation}
for every $\zeta\in \C^\infty_0(\R^2;\R^2)$, and
\begin{equation}\label{ELeq2}
  \int_{\RR^2} \left( - \Delta \phi_h^n + \alpha \phi_h^n - \rho_h^n \right)\eta \dd x 
 = \tau \int_{\RR^2} \frac{\phi_h^{n-1} -\phi_h^n}{{h}} \eta \dd x ,
\end{equation}
for every  $\eta\in \C^\infty_0(\R^2)$.
Moreover, the following identities are satisfied:
\begin{equation}\label{discreteW2}
  \int_{\R^2}\left|\frac{\nabla\rho_h^n}{\rho_h^n}-\chi\nabla\phi_h^n \right|^2\rho_h^n \dd x =
  \frac{d_W^2(\rho_h^n,\rho_h^{n-1})}{h^2},
\end{equation}
and
\begin{equation}\label{discreteL2}
    \|\Delta\phi_h^n -\alpha\phi_h^n+\rho_h^n \|_2^2 =
    \tau^2 \frac{\|\phi_h^n-\phi_h^{n-1}\|_{{\rm L}^2(\R^2)}^2}{h^2}.
\end{equation}
Finally, the approximative weak solution estimate 
\begin{equation}
\left| \int_{\RR^2} \left[ \xi (\rho_h^n-\rho_h^{n-1}) + h\ \nabla\xi\cdot \left( \nabla\rho_h^n - \chi\ \rho\ \nabla\phi_h^n \right) \right] \dd x \right| \le \|\xi\|_{\rm W^{2,\infty}} \frac{d_W^2(\rho_h^n,\rho_h^{n-1})}{2}\ \label{bb40}
\end{equation}
holds for any $\xi \in \C^\infty_0(\RR^2)$.
\end{lemma}
\begin{proof}
In order to simplify the notation, in this proof we use the notation
$\rho=\rho_h^n$, $\bar\rho=\rho_h^{n-1}$, $\phi=\phi_h^n$, $\bar\phi=\phi_h^{n-1}$.   

Let $\zeta\in\C^\infty_0(\RR^2;\RR^2)$ and $\eta\in\C^\infty_0(\RR^2)$ be two smooth functions. Define $\TT_\delta:=\id + \delta\,\zeta$ and for $\delta\in (0,1)$,
\begin{equation*}
 \rho_\delta:= (\TT_\delta) _\# \rho\;, \quad \phi_\delta := \phi + \delta\, \eta \;.
\end{equation*}

\noindent $\bullet$ It is standard, see~\cite[Theorem~5.30]{Vi03} for instance, that
\begin{equation}
\label{bb46}
\lim_{\delta\to 0} \frac1\delta \int_{\RR^2}\left(\rho_\delta \log \rho_\delta - \rho \log \rho\right)\dd x = - \int_{\RR^2}\Delta \zeta(x)\,\rho(x)\dd x\;.
\end{equation}

\noindent $\bullet$ It is also classical, see~\cite[Theorem~8.13]{Vi03} for instance, that
\begin{equation}
 \lim_{\delta \to 0} \frac{d^2_W(\rho_\delta ,\bar \rho)-d^2_W(\rho ,\bar \rho)}{2\delta}
 = - \int_{\RR^2} \left[(\id-T_{\bar\rho}^{\rho})\cdot (\zeta\!\circ\!T_{\bar\rho}^{\rho})\right] \,\bar \rho\dd x\,, \label{bb49}
\end{equation}
where $T_{\bar\rho}^{\rho}$ is the optimal map pushing $\bar \rho$ onto $\rho$.

\noindent $\bullet$ A standard computation gives
\begin{multline}
\lim_{\delta\to 0} \frac{1}{2\delta}\ \left[ \|\nabla \phi_\delta\|_{{\rm L}^2(\R^2)}^2 + 
\alpha\, \|\phi_\delta\|_{{\rm L}^2(\R^2)}^2-\|\nabla \phi\|_{{\rm L}^2(\R^2)}^2 - \alpha\, \|\phi\|^2_{{\rm L}^2(\R^2)} \right] \\= \int_{\RR^2} \left( \nabla \phi\cdot\nabla \eta + \alpha \phi\, \eta \right) \dd x
= \int_{\RR^2} \left( -\Delta \phi + \alpha \phi\right)\, \eta  \dd x\;. \label{bb47}
\end{multline}

\noindent $\bullet$ Since $\phi \in \H^1(\R^2)$, we have%
\begin{equation*}
  \frac{\phi\!\circ\!\TT_\delta -\phi}{\delta} \rightharpoonup \zeta \cdot \nabla \phi\quad \mbox{in $\LL^2(\RR^2)$,} \qquad \eta\!\circ\!\TT_\delta  \rightarrow \eta \quad \mbox{in $\LL^2(\RR^2),$}
\end{equation*}
and recalling that $\rho \in \LL^2(\RR^2)$ by Lemma~\ref{impro}, we conclude that
\begin{equation}
\label{bb48}
  \begin{aligned}
\frac1\delta\int_{\RR^2} [\rho\,\phi-\rho_\delta\,\phi_\delta](x) \dd x &=  \frac1\delta\int_{\RR^2} \rho \left[\phi-\phi\!\circ\!\TT_\delta)-\delta\, \eta\!\circ\!\TT_\delta \right]\dd x\\
&\mathop{\longrightarrow}_{\delta \to 0} - \int_{\RR^2} \left( \zeta \cdot \nabla \phi + \eta \right) \rho\, \dd x\;.
  \end{aligned}
\end{equation}

\noindent $\bullet$ We then infer from~\eqref{bb46}, \eqref{bb49}, \eqref{bb47}, and~\eqref{bb48} that
\begin{equation*}
  \begin{aligned}
0 &\le \lim_{\delta\to 0} \frac1\delta \left( \FF[\rho_\delta,\phi_\delta]-\FF[\rho,\phi]\right)\\
&= -\frac1{h \chi} \int_{\RR^2}(x-T_{\bar\rho}^{\rho})\cdot (\zeta\!\circ\!T_{\bar\rho}^{\rho})\,\bar \rho\dd x + \frac{\tau}{h}\int_{\RR^2} \eta\,(\phi-\bar \phi) \dd x\\
&\qquad-\frac1\chi \int_{\RR^2}\Delta \zeta\,\rho \dd x-\int_{\RR^2} \rho\, \zeta\cdot \nabla \phi\dd x \\
&\qquad -  \int_{\RR^2}\rho\, \eta \dd x+ \int_{\RR^2} [-\Delta \phi +\alpha \,\phi] \, \eta\dd x\;.
  \end{aligned}
\end{equation*}
The above inequality being valid for arbitrary $(\zeta,\eta)\in\C_0^\infty(\RR^2;\RR^2)\times\C_0^\infty(\RR^2)$, it is also valid for $(-\zeta,-\eta)$ so that we end up with
\begin{eqnarray}
& & \frac{1}{\chi} \int_{\RR^2} \zeta\cdot \left( \nabla \rho - \chi\rho\nabla \phi \right) \dd x + \int_{\RR^2} \left( \tau\ \frac{\phi-\bar \phi}{h} - \Delta \phi + \alpha \phi - \rho \right)\ \eta \dd x \label{bb50}\\
& &=  \frac1{h \chi} \int_{\RR^2}(x-T_{\bar\rho}^{\rho})\cdot (\zeta\!\circ\!T_{\bar\rho}^{\rho})\,\bar \rho\dd x\,. \nonumber
\end{eqnarray}

Taking $\zeta=0$ in~\eqref{bb50} we obtain \eqref{ELeq2}. While choosing $\eta=0$ in~\eqref{bb50} gives, for all $\zeta\in\C_0^\infty(\RR^2;\RR^2)$
\begin{equation}
\int_{\RR^2} \zeta\cdot \left( \nabla \rho- \chi\rho\nabla \phi \right) \dd x = \frac{1}{h} \int_{\RR^2}(x-T_{\bar\rho}^{\rho})\cdot (\zeta\!\circ\!T_{\bar\rho}^{\rho})\,\bar \rho\dd x\;, \label{bb51}
\end{equation}
and \eqref{ELeq1} follows from \eqref{ot} and the fact that $T^\rho_{\bar \rho}$ pushes $\bar\rho$ onto $\rho$.

In order to obtain \eqref{discreteW2}, we observe that
$\nabla\phi_h^n\in\L^4(\R^2)$ as a consequence of the regularity $\phi_h^n\in \H^2(\R^2)$ established in Lemma \ref{impro} and the continuous embedding of $\H^2(\R^2)$ in $W^{1,4}(\R^2)$. Since $\rho_h^n\in\L^2(\R^2)$ we conclude that $\nabla\phi_h^n\in\L^2(\rho_h^n)$.
From \eqref{ELeq1} it follows that
$$\frac{\nabla\rho_h^n}{\rho_h^n}-\chi\nabla\phi_h^n  = \frac1h (T_{\rho_h^n}^{\rho_h^{n-1}}-\id), \qquad\mbox{ in }\L^2(\rho_h^n).$$
The equality of the $\L^2(\rho_h^n)$ norms yields \eqref{discreteW2} after using the properties and the definition of optimal transport $T_{\bar\rho}^\rho$.
Identity \eqref{discreteL2} follows immediately by \eqref{ELeq2}. 

Finally consider $\xi \in \CC_0^\infty(\RR^2)$. By the Taylor expansion, we have, for $x\in\RR^2$
\begin{equation*}
\left| \xi(x) - \xi (T_{\bar\rho}^{\rho}(x)) - (\nabla\xi\circ T_{\bar\rho}^{\rho})(x)\cdot (x-T_{\bar\rho}^{\rho}(x)) \right| \le \|D^2\xi\|_{{\rm L}^\infty(\R^2)}\ \frac{|x-T_{\bar\rho}^{\rho}(x)|^2}{2}\;.
\end{equation*}
Multiplying by $\bar \rho$ and integrating over $\RR^2$ gives
$$
\left| \int_{\RR^2} \left[ \xi\, \bar \rho - \xi\, \rho - (\nabla\xi\circ T_{\bar\rho}^{\rho})\cdot (\id-T_{\bar\rho}^{\rho})\, \bar \rho \right] \dd x \right| \le \|D^2\xi\|_{{\rm L}^\infty(\R^2)} \frac{d_W^2(\rho,\bar \rho)}2\,.
$$
Combining the above inequality with \eqref{bb51} (with $\zeta=\nabla\xi$) leads us to~\eqref{bb40}.
\end{proof}
\section{Convergence}
\subsection{One-step estimates}
\begin{lemma}[Uniform estimates]\label{lemma7} There exists a constant $C_4>0$ such that, for all $h,T>0$ and $N \ge 1$ satisfying $N h \le T$,
\begin{multline*}
  \frac{1}{16\chi T} \int_{\R^2}|x|^2 \rho_h^N \dd x+ \frac{1}{4\chi h}\sum_{n=0}^{N-1}d_W^2(\rho_h^{n+1},\rho_h^{n})
+\frac{\tau}{2h}\sum_{n=0}^{N-1}\|\phi_h^{n+1}-\phi_h^{n}\|_{{\rm L}^2(\R^2)}^2\\+\frac{8\pi - \chi}{16 \pi \chi} \int_{\RR^2}  \rho_h^N |\log \rho_h^N| \dd x +\nu \left[\|\nabla \phi_h^N\|_{{\rm L}^2(\R^2)}^2 + \alpha \|\phi_h^N\|_{{\rm L}^2(\R^2)}^2 \right] \\\le C_4 + \left(1+\frac1T+(\log T)_+\right)
\end{multline*}
where $\nu$ is defined in the proof of Lemma~{\rm\ref{tintin}}.
\end{lemma}
\begin{proof}
For $n \ge 0$, $\FF[\rho_h^{n+1},\phi_h^{n+1}] \le \FF[\rho_h^{n},\phi_h^{n}]$, so that
\begin{equation*}
  \frac{1}{2\chi h}d_W^2(\rho_h^{n+1},\rho_h^{n}) + \frac{\tau}{2h}\|\phi_h^{n+1}-\phi_h^{n}\|_{{\rm L}^2(\R^2)}^2 + \EE[\rho_h^{n+1},\phi_h^{n+1}] \le \EE[\rho_h^{n},\phi_h^{n}]\;.
\end{equation*}
Summing up over $n \in \{0,\cdots,N-1\}$, we find
\begin{multline}\label{eq:c1}
  \frac{1}{2\chi h}\sum_{n=0}^{N-1}d_W^2(\rho_h^{n+1},\rho_h^{n}) + \frac{\tau}{2h}\sum_{n=0}^{N-1}\|\phi_h^{n+1}-\phi_h^{n}\|_{{\rm L}^2(\R^2)}^2 + \EE[\rho_h^{N},\phi_h^{N}] \le \EE[\rho_{0},\phi_{0}]\;.
\end{multline}
By Cauchy-Schwarz' inequality, we deduce that
\begin{equation}\label{eq:c2}
    d_W^2(\rho_h^{N},\rho_h^{0}) \le \left[\sum_{n=0}^{N-1}d_W(\rho_h^{n+1},\rho_h^{n})\right]^2 
\le \frac{T}{h} \sum_{n=0}^{N-1}d_W^2(\rho_h^{n+1},\rho_h^{n})\;.
\end{equation}
We thus infer from~\eqref{eq:c1},~\eqref{eq:c2}, and the lower bound~\eqref{eq:38} for $\EE$ that
\begin{align} 
\EE[\rho_{0},\phi_{0}]\ge &\, \frac{1}{4\chi T}d_W^2(\rho_h^{N},\rho_h^{0})+\frac{1}{4\chi h}\sum_{n=0}^{N-1}d_W^2(\rho_h^{n+1},\rho_h^{n})+\frac{\tau}{2h}\sum_{n=0}^{N-1}\|\phi_h^{n+1}-\phi_h^{n}\|_2^2\nonumber \\
&\,+\frac{8\pi - \chi}{16 \pi\chi} \int_{\RR^2}  \rho_h^N |\log \rho_h^N| \dd x +\nu \left[\|\nabla \phi_h^N\|_{{\rm L}^2(\R^2)}^2 + \alpha \|\phi_h^N\|_{{\rm L}^2(\R^2)}^2 \right] \nonumber \\
&\,+ \frac3{2\chi}\int_{\R^2}\rho_h^N\log H\dd x-C_1 \;.\label{p13}
\end{align}
Since the triangle inequality implies that
\begin{align*}
    \int_{\R^2}|x|^2 \rho_h^N \dd x = d_W^2(\rho_h^N,\delta_0) &\le 2 d_W^2(\rho_h^N,\rho_h^0) + 2d_W^2(\rho_h^0,\delta_0) \\
&= 2 d_W^2(\rho_h^N,\rho_h^0) +2\int_{\R^2}|x|^2 \rho_0 \dd x \;,
  \end{align*}
it follows that Equation~\eqref{p13} results in
\begin{align*}
       \frac3\chi \!\int_{\R^2}\rho_h^N\log(1+|x|^2)\dd x  
\ge&\, -\bar C_4 \left(1+\frac1T\right)+ \frac{1}{8\chi T} \int_{\R^2}|x|^2 \rho_h^N \dd x\\ &\,
+ \frac{1}{4\chi h}\sum_{n=0}^{N-1}d_W^2(\rho_h^{n+1},\rho_h^{n}) + \frac{\tau}{2h}\sum_{n=0}^{N-1}\|\phi_h^{n+1}-\phi_h^{n}\|_{{\rm L}^2(\R^2)}^2\\
  &\,+\frac{8\pi - \chi}{16 \pi\chi} \int_{\RR^2}  \rho_h^N |\log \rho_h^N| \dd x +\nu \left[\|\nabla \phi_h^N\|_{{\rm L}^2(\R^2)}^2 + \alpha \|\phi_h^N\|_{{\rm L}^2(\R^2)}^2 \right] 
    \end{align*}
where 
$$\bar C_4:=\EE[\rho_{0},\phi_{0}]+C_1+\frac{3\log \pi}{2\chi}-\frac1{4\chi}\int_{\R^2}|x|^2\rho_0\dd x\;.$$
Since $\log(1+|x|^2) \le \eps |x|^2 + (-\log \eps)_+$ for all $\eps >0$ and $x\in\R^2$, we obtain
\begin{multline*}
\bar C_4 + \frac{3\eps}{\chi} \int_{\R^2}|x|^2 \rho_h^N \dd x + \frac{3}{\chi}(-\log \eps)_+ \ge  \frac{1}{8\chi T} \int_{\R^2}|x|^2 \rho_h^N \dd x+ \frac{1}{4\chi h}\sum_{n=0}^{N-1}d_W^2(\rho_h^{n+1},\rho_h^{n})\\
+\frac{\tau}{2h}\sum_{n=0}^{N-1}\|\phi_h^{n+1}-\phi_h^{n}\|_{{\rm L}^2(\R^2)}^2+\frac{8\pi - \chi}{16 \pi \chi} \int_{\RR^2}  \rho_h^N |\log \rho_h^N| \dd x\\ +\nu \left[\|\nabla \phi_h^N\|_{{\rm L}^2(\R^2)}^2 + \alpha \|\phi_h^N\|_{{\rm L}^2(\R^2)}^2 \right]  \;.
\end{multline*}
Taking $\eps:={1}/{(48 T)}$ and $C_4=\bar C_4 +\frac{3}{\chi}\log 48$ we obtain the desired bound.
\end{proof}
\subsection{Estimates on the interpolant}
We consider the piecewise constant time dependent pair of functions $(\rho_h,\phi_h)$ defined by
\begin{equation*}
(\rho_h(t),\phi_h(t)) := (\rho_{h}^n,\phi_{h}^n)\,, \qquad t\in ((n-1)h,nh]\,, \qquad n\ge 0\,.
\end{equation*}
\begin{lemma}[Time integrated estimates]\label{phiover}
 Let $T>0$. There exists a constant $C_5(T)>0$ such that, for all $h>0$ and $N \ge 1$ satisfying $Nh \le T$ it holds
  \begin{equation*}
    \int_0^{Nh} \left( \int_{\R^2}\left|\frac{\nabla\rho_h(s)}{\rho_h(s)}\right|^2\rho_h(s)\dd x + \| \Delta \phi_h(s) + \rho_h(s) - \alpha \phi_h(s) \|_{{\rm L}^2(\R^2)}^2\right) \dd s \le C_5(T)\;.
  \end{equation*}
\end{lemma}
\begin{proof}
Fix $N \ge 1$ such that $N h \le T$. We set
$$\Lambda(T)=16 + 2\log\pi+ \frac2\ee + C_4\left( \frac{16 \pi \chi}{8\pi - \chi} + 16\chi\,T\right)\left( 1 + \frac1T+ (\log T)_+\right)\;.$$
By Lemma~\ref{lemma7}, for $n \in \{1,\cdots,N\}$ we obtain
\begin{multline}\label{banff}   
\int_{\R^2} \rho_h^n \log \rho_h^n\dd x + 4 \int_{\R^2}  \rho_h^n \log (1+|x|^2)\dd x \le \int_{\R^2} \rho_h^n | \log \rho_h^n|\dd x + 4 \int_{\R^2} |x|^2 \rho_h^n \dd x \\
\le  \left(\frac{16\pi\chi}{8\pi - \chi}+16\chi T\right)C_4\left( 1 + \frac1T+(\log T)_+ \right)=\Lambda(T) -16 - \frac2\ee - 2\log\pi\;.
\end{multline}
We then infer from Lemma~\ref{impro} that, for $n \in \{0,\cdots,N-1\}$
\begin{align*}
    \frac1\chi  \int_{\R^2}\left|\frac{\nabla\rho_h^{n+1}}{\rho_h^{n+1}}\right|^2\rho_h^n&\dd x
         + \|\Delta \phi_h^{n} + \rho_h^{n+1} -\alpha\phi_h^{n+1}\|_{{\rm L}^2(\R^2)}^2\\
\le&\, \frac{2}{\chi h} \left[ \int_{\R^2} \rho_h^{n} \log \rho_h^{n}\dd x - \int_{\R^2} \rho_h^{n+1} \log \rho_h^{n+1}\dd x\right] \\
&\,+ \frac{\tau}{h}  \left[ \|\nabla \phi_h^n\|_{{\rm L}^2(\R^2)}^2 + \alpha  \| \phi_h^n\|_{{\rm L}^2(\R^2)}^2 -\|\nabla \phi_h^{n+1}\|_{{\rm L}^2(\R^2)}^2 - \alpha  \| \phi_h^{n+1}\|_{{\rm L}^2(\R^2)}^2 \right]\\ &+C(\Lambda(T)) 
\,+ \alpha  \| \phi_h^{n+1}\|_{{\rm L}^2(\R^2)}^2 \;.
  \end{align*}
Summing over $n \in \{0,\cdots,N-1 \}$ gives
\begin{align*}
    \frac1\chi\sum_{n=0}^{N-1}\int_{\R^2}\left|\frac{\nabla\rho_h^{n+1}}{\rho_h^{n+1}}\right|^2&\rho_h^n\dd x + \sum_{n=0}^{N-1}\|\Delta \phi_h^{n+1} + \rho_h^{n+1} -\alpha\phi_h^{n+1}\|_{{\rm L}^2(\R^2)}^2\\
\le&\, \frac{2}{\chi h} \left[ \int_{\R^2} \rho_0 \log \rho_0\dd x - \int_{\R^2} \rho_h^{N} \log \rho_h^{N}\dd x\right]
\\
&\,+ \frac{\tau}{h}  \left[ \|\nabla \phi_0\|_{{\rm L}^2(\R^2)}^2 + \alpha  \| \phi_0\|_{{\rm L}^2(\R^2)}^2 -\|\nabla \phi_h^{N}\|_{{\rm L}^2(\R^2)}^2 - \alpha  \| \phi_h^{N}\|_{{\rm L}^2(\R^2)}^2 \right] \\ &+NC(\Lambda(T)) 
\,+ \alpha \sum_{n=0}^{N-1} \| \phi_h^{n+1}\|_{{\rm L}^2(\R^2)}^2 \;.
  \end{align*}
Therefore, using once more Lemma~\ref{lemma7} together with~\eqref{banff}, we conclude that
\begin{align*}
    \frac1\chi\int_{0}^{Nh} \int_{\R^2}\left|\frac{\nabla\rho_h(s)}{\rho_h(s)}\right|^2\rho_h(s)\dd x  \dd s &+\int_{0}^{Nh} \|\Delta \phi_h(s) + \rho_h(s) -\alpha\phi_h(s)\|_{{\rm L}^2(\R^2)}^2 \dd s\\
\le &\, \frac{2}{\chi} \left[ \int_{\R^2} \rho_0 \log \rho_0\dd x +\Lambda(T)\right]
+ \alpha N h \!\sup_{n \in [1,N]}\!\| \phi_h^{n}\|_{{\rm L}^2(\R^2)}^2\\
&\,+ \tau  \left[ \|\nabla \phi_0\|_{{\rm L}^2(\R^2)}^2 + \alpha  \| \phi_0\|_{{\rm L}^2(\R^2)}^2 \right]+N h C(\Lambda(T)) 
\\ \le&\, C(T)+\alpha T \frac{C_4\left(1+1/T + (\ln T)_+\right)}{\alpha \nu} \le C_5(T)\;,
  \end{align*}
which completes the proof.
\end{proof}

\subsection{De Giorgi interpolant and Discrete energy dissipation}

In order to obtain an energy dissipation estimate we introduce the so called De Giorgi variational interpolant 
(see for instance \cite[Section 3.2]{AGS08}). We define the De Giorgi interpolant as follows
\begin{equation*}
    \tilde u_h(t) \in  {\rm Argmin}_{u\in X}\left\{\frac{1}{2(t-(n-1)h)}d^2(u,u_h^{n-1}) +\EE(u)\right\}, \qquad t\in ((n-1)h,nh].
\end{equation*}
We can also assume that $\tilde u_h(nh)=u_h^n$ for any $n\in\NN$.
We use the notation $(\tilde\rho_h(t),\tilde\phi_h(t))=\tilde u_h(t)$.
\begin{proposition}
For every $t>0$, $(\tilde\rho_h(t),\tilde\phi_h(t))$ enjoy the same regularity properties of $(\rho^n_h,\phi^n_h)$ given by Lemma {\rm\ref{impro}}
and the following discrete energy identity holds for all $N\in\NN$ and $h>0$
\begin{equation}\label{discreteEI}
\begin{aligned}
 &\frac{1}{2\chi}\int_0^{Nh}\int_{\R^2}\left|\frac{\nabla\rho_h}{\rho_h}-\chi\nabla\phi_h \right|^2\rho_h\,\dd x\,\dd t
+ \frac{\tau}{2}\int_0^{Nh}\|\Delta\phi_h -\alpha\phi_h+\rho_h \|_{{\rm L}^2(\R^2)}^2 \,\dd t \\
&+ \frac{1}{2\chi}\int_0^{Nh}\int_{\R^2}\left|\frac{\nabla\tilde\rho_h}{\tilde\rho_h}-\chi\nabla\tilde\phi_h \right|^2\tilde\rho_h\,\dd x\,\dd t
+ \frac{\tau}{2}\int_0^{Nh}\|\Delta\tilde\phi_h -\alpha\tilde\phi_h+\tilde\rho_h \|_{{\rm L}^2(\R^2)}^2 \,\dd t  \\
& +\EE[\rho_h({Nh}),\phi_h({Nh})] = \EE[\rho_0,\phi_0].
\end{aligned}
\end{equation}
Moreover for every $T>0$ there exists a constant $C(T)$ such that
\begin{equation}\label{DGversusPC}
    d^2(\tilde u_h(t),u_h(t))\leq C(T) h, \qquad \forall t\in[0,T].
\end{equation}
\end{proposition}
\begin{proof}
From \cite[Lemma~3.2.2]{AGS08} we have the one step energy identity
\begin{equation*}
\begin{aligned}
 &\frac{1}{2}\frac{d^2(u_h^n,u_h^{n-1})}{h}
+ \frac{1}{2}\int_{(n-1)h}^{nh}\frac{d^2(\tilde u(t),u_h^{n-1})}{t-(n-1)h} \dd t
 +\EE(u_h^n)= \EE(u_h^{n-1}).
\end{aligned}
\end{equation*}
Defining the function
\begin{equation*}
    G_h(t) = \frac{d(\tilde u(t),u_h^{n-1})}{t-(n-1)h}, \qquad t\in ((n-1)h,nh],
\end{equation*}
and summing from $n=1$ to $N$, we obtain
\begin{equation}\label{discreteEII}
\begin{aligned}
 \frac{1}{2}\sum_{n=1}^N h\frac{d^2(u_h^n,u_h^{n-1})}{h^2}
+ \frac{1}{2}\int_{0}^{Nh} G^2_h(t) \dd t
 +\EE(u_h^N)= \EE(u_0).
\end{aligned}
\end{equation}
The same argument of Lemma \ref{impro} shows that $(\tilde\rho_h(t),\tilde\phi_h(t))$ enjoy the same regularity properties of 
$(\rho_h^n,\phi_h^n)$ and we can obtain  
the Euler-Lagrange equation for $(\tilde\rho_h(t),\tilde\phi_h(t))$:
\begin{equation*}
\begin{aligned}
& \frac{1}{\chi} \int_{\RR^2} \zeta\cdot \left( \nabla\tilde\rho_h(t) - \chi\tilde\rho_h(t)\nabla \tilde\phi_h(t) \right) \dd x
 + \int_{\RR^2} \left( - \Delta \tilde\phi_h(t) + \alpha \tilde\phi_h(t) - \tilde\rho_h(t) \right)\ \eta \dd x \\
& =  \frac1{\chi} \frac1{{t-(n-1)h}}\int_{\RR^2}(T_{\tilde\rho_h(t)}^{\rho_h^{n-1}}-Id)\cdot \zeta \tilde\rho_h(t)\dd x\,
 + \tau \int_{\RR^2} \frac{\phi_h^{n-1} -\tilde\phi_h(t)}{{t-(n-1)h}} \ \eta \dd x ,
\end{aligned}
\end{equation*}
for every $\zeta\in \C^\infty_0(\R^2;\R^2)$ and  $\eta\in \C^\infty_0(\R^2)$.
As in Lemma \ref{lem:b7} it follows that
\begin{equation}\label{discreteW}
  \int_{\R^2}\left|\frac{\nabla\tilde\rho_h(t)}{\tilde\rho_h(t)}-\chi\nabla\tilde\phi_h(t) \right|^2\tilde\rho_h(t) \dd x =
  \frac{d_W^2(\tilde\rho_h(t),\rho_h^{n-1})}{(t-(n-1)h)^2}
\end{equation}
and
\begin{equation}\label{discreteL}
    \|\Delta\tilde\phi_h(t) -\alpha\tilde\phi_h(t)+\tilde\rho_h(t) \|_{{\rm L}^2(\R^2)}^2 =
    \tau^2 \frac{\|\tilde\phi_h(t)-\phi_h^{n-1}\|_{{\rm L}^2(\R^2)}^2}{(t-(n-1)h)^2}
\end{equation}
for $t\in ((n-1)h,nh]$.
Recalling the definition of $d$ and using the identities~\eqref{discreteW},~\eqref{discreteL},~\eqref{discreteW2}, and~\eqref{discreteL2} 
in \eqref{discreteEII} we obtain \eqref{discreteEI}.

Finally, the estimate \eqref{DGversusPC} follows from Lemma \ref{lemma7} using the same argument of \cite[Lemma~3.2.2]{AGS08}.
\end{proof}

\begin{lemma}[Time equicontinuity]\label{london}
Let $T>0$. There exist $C_6$ and $C_7$ such that for all $(t,s) \in [0,T]^2$ and $h \in (0,1)$, we get
\begin{eqnarray*}
d_W(\rho_h(t),\rho_h(s)) & \le & C_6(T) (\sqrt{|t-s|}+\sqrt{h})\;, \\
\|\phi_h(t)-\phi_h(s)\|_{{\rm L}^2(\R^2)} & \le & C_7(T) (\sqrt{|t-s|}+\sqrt{h})\;.
\end{eqnarray*}
\end{lemma}
\begin{proof}
Let $0\le s <t$ and set $N:=\left\lceil t/h\right\rceil$ and $P:=\left\lceil s/h\right\rceil$, where $\lceil a\rceil$ denotes the superior integer part of the real number $a$.
By Lemma~\ref{lemma7}, we deduce that
\begin{equation*}
\begin{aligned}
d(u_h(t),u_h(s)) =d(u_h^N,u_h^P) &\le \sum_{n=P}^{N-1} d(u_h^n,u_h^{n+1})\le \sqrt{N-P}\sqrt{\sum_{n=P}^{N-1} d^2(u_h^n,u_h^{n+1})}\\
&\le \sqrt{N-P}\,\sqrt{2h(C_4 + (\log T)_+)}\;,
\end{aligned}
\end{equation*}
which gives the time equi-continuity for $\rho_h$ and $\phi_h$
recalling the definition \eqref{def:d} of the distance $d$.
\end{proof}

\subsection{Proof of Theorems~\ref{main1} and~\ref{main2}}

\begin{proof} Let $T>0$.

\noindent $\bullet$ Convergence of $(\rho_h)_h$ and $(\phi_h)_h$:
By Lemma~\ref{lemma7}, we obtain
\begin{equation}\label{boundH1}
\sup_{t\in[0,T],h\in(0,1)}\|\phi_h(t)\|_{\HH^1(\R^2)} <+\infty.
\end{equation}
Thus $\{\phi_h(t): (t,h) \in [0,T]\times (0,1)\}$ is in a weakly compact subset of $\HH^1(\R^2)$. 
Also, Lemma~\ref{lemma7} implies that
\begin{equation}\label{momententropy}
\sup_{t\in[0,T],h\in(0,1)}
\left[
\int_{\R^2}|x|^2\rho_h(t)\dd x +\int_{\R^2}\rho_h(t)|\log\rho_h(t)|\dd x
\right] <+\infty.
\end{equation}
Hence, the set $\{\rho_h(t): (t,h) \in [0,T]\times (0,1)\}$ is tight and equi-integrable. Thus, by Dunford-Pettis theorem, this set is weakly compact in $\L^1(\R^2)$.

In addition, the equi-continuity stated in Lemma~\ref{london} guarantees that, for all $0\le s\le t\le T$,
\begin{equation*}
\limsup_{h \to 0} \|\phi_h(t)-\phi_h(s) \|_{{\rm L}^2(\R^2)}^2 \le \, C_7(T)\sqrt{t-s}\,,
\end{equation*}
and
\begin{equation*}
\limsup_{h \to 0} d_W(\rho_h(t),\rho_h(s)) \le \, C_6(T)\sqrt{t-s}\,.
\end{equation*}
We then infer from a variant of the Ascoli-Arzel\`a theorem \cite[Proposition~3.3.1]{AGS08} that
there exist a monotone sequence $(h_j)_j$ of positive numbers, $h_j\to 0$, and curves
$$
\phi\in \CC^{1/2}([0,T];\L^2(\R^2)), \qquad \rho\in \CC^{1/2}([0,T];\PP_2(\R^2)),
$$
such that
\begin{equation*}
\phi_{h_j}(t) \rightharpoonup \phi(t)\quad\mbox{ weakly in $\HH^1(\R^2)$ for all $t\in [0,T]$}
\end{equation*}
and
\begin{equation*}
\rho_{h_j}(t) \rightharpoonup \rho(t)\quad\mbox{ weakly in $\L^1(\R^2)$ for all $t\in [0,T]$.}
\end{equation*}
Passing to the limit as $h_j\to 0$ in \eqref{momententropy} and in \eqref{boundH1}, by semicontinuity we obtain the bounds in \eqref{ME}.

\noindent$\bullet$
Moreover, Lemma~\ref{phiover} implies that $(\phi_h)_h$ is bounded in $\L^2(0,T;\HH^2(\R^2))$.
We can assume without lose of generality,
\begin{equation}
  \label{cvphi3}
  \phi_{h_j} \rightharpoonup \phi\quad\mbox{weakly in $\L^2(0,T;\HH^2(\R^2))$}
\end{equation}
and
$$
\phi_{h_j} \longrightarrow \phi\quad\mbox{in $\L^2(0,T;\L^2_{\rm loc}(\R^2))$}.
$$
By standard interpolation results we obtain that
\begin{equation}
\phi_{h_j} \longrightarrow \phi\quad\mbox{in $\L^2(0,T;\HH^1_{\rm loc}(\R^2))$.} \label{cvphi2}
\end{equation}

\noindent$\bullet$
Lemma~\ref{phiover} also implies that $(\|{\nabla\rho_{h_j}}/{\rho_{h_j}}\|_{\L^2(\rho_{h_j})})_j$ is bounded in $\L^2(0,T)$. Then, from the inequality \eqref{BHN} and the second bound in~\eqref{momententropy}, we obtain that $({\rho_{h_j}})_j$ is bounded in $\L^2((0,T)\times \R^2)$.
We deduce that, after extracting a subsequence,
\begin{equation}\label{luego}
\rho_{h_j} \rightharpoonup \rho\quad\mbox{ weakly in $\L^2(0,T;\R^2)$.}
\end{equation}

In order to pass to the limit in $\nabla \rho_{h_j}$, we use Proposition~\ref{prop:vf}
with the measures $\dd\mu_j=\rho_{h_j} \!\dd x \!\dd t/T$ in the space $(0,T)\times\R^2$
and the vector fields $v_j={\nabla\rho_{h_j}}/{\rho_{h_j}}$.
By Lemma~\ref{phiover}, we have 
$$
\sup_{j}\int_0^T\int_{\R^2} |v_j|^2 \rho_{h_j} \dd x \dd t <+\infty \;.
$$
Setting $\!\dd\mu= \rho \!\dd x \!\dd t/T$, there exists $v\in \L^2((0,T)\times\R^2,\mu;\R^2)$
(consequently $v \rho \in \L^1((0,T)\times \R^2)$)
such that, up to a subsequence,
\begin{equation*}
    \int_0^T\int_{\R^2} \varphi\cdot v_j \, \rho_{h_j} \dd x \dd t \to \int_0^T\int_{\R^2} \varphi\cdot v \, \rho \dd x \dd t,
\end{equation*}
for every $\varphi\in \C^\infty_0((0,T)\times \R^2)$. 
Since $v \rho \in \L^1((0,T)\times \R^2)$, we can deduce
\begin{align*}
    \int_0^T\int_{\R^2} \varphi\cdot v_j \,\rho_{h_j} \dd x \dd t \,& =
     \int_0^T\int_{\R^2} \varphi\cdot \nabla\rho_{h_j} \dd x \dd t \\ &=
     - \int_0^T\int_{\R^2} (\nabla\cdot\varphi) \,\rho_{h_j} \dd x \dd t
     \to - \int_0^T\int_{\R^2} (\nabla\cdot\varphi) \,\rho \dd x \dd t\,.
     \end{align*}
Consequently,  $v \rho = \nabla\rho $ and $\rho\in\L^1(0,T;W^{1,1}(\R^2))$.
Finally, the lower semicontinuity property \eqref{VFlsc} yields~\eqref{B}.

\noindent$\bullet$ Identifying the limit:
Writing the  Euler-Lagrange equation, see Lemma~\ref{lem:b7}, with a time dependent test function, we obtain a time discrete formulation of the system~\eqref{num2}. 
Thanks to the convergences \eqref{cvphi3}-\eqref{cvphi2} for $(\phi_h)_h$, the convergence~\eqref{luego} for $(\rho_h)_h$
and the previous step for $(\nabla\rho_h)_h$, we can pass to the limit in this  time discrete formulation
and conclude that $(\rho,\phi)$ is a weak solution to the Keller-Segel system~\eqref{num2}.

\noindent $\bullet$ Energy inequality:
We first recall that De Giorgi interpolant converges to the same limit as the piecewise constant interpolant, see \eqref{DGversusPC}. This fact together with the above compactness properties, Proposition~\ref{prop:vf},
and the lower semicontinuity of $\EE$, we can pass to the limit in the discrete energy identity~\eqref{discreteEI} obtaining the energy inequality \eqref{energy inequality}.
\end{proof}

 \appendix
\section{Biler-Hebisch-Nadzieja inequality}
A similar inequality is proved in~\cite{BHN92}.
\begin{lemma}[Biler-Hebisch-Nadzieja inequality]
  Given $\eps >0$, there is $L_\eps >0$ such that for all non-negative $f \in \HH^1(\R^2)$ satisfying $f^2 \log  f \in \LL^1(\R^2)$
  \begin{equation}\label{eqaBHN}
    \|f\|_{{\rm L}^4(\R^2)}^4 \le \eps \|\nabla f\|_{{\rm L}^2(\R^2)}^2 \|f^2 \log  f\|_{{\rm L}^1(\R^2)}  + L_\eps \|f\|_{{\rm L}^2(\R^2)}^2 \;.
  \end{equation}
\end{lemma}
\begin{proof}
  For $N>1$ define
  \begin{equation*}
  \Theta_N(s):=  \left\{
      \begin{array}{ll}
        0&\quad\mbox{if $s<N$}\\
        2(s-N)&\quad\mbox{if $N \le s \le 2N$}\\
        s&\quad\mbox{if $s>N$}
      \end{array}
\right.
  \end{equation*}
By Gagliardo-Nirenberg's inequality
\begin{eqnarray*}
  \|f\|_{{\rm L}^4(\R^2)}^4 &=&  \|f - \Theta_N (f) + \Theta_N (f)\|_{{\rm L}^4(\R^2)}^4 \le C \|\Theta_N (f) \|_{{\rm L}42(\R^2)}^4+C \|f - \Theta_N (f) \|_{{\rm L}^4(\R^2)}^4  \\
&\le& C\|\nabla\Theta_N (f)\|_{{\rm L}^2(\R^2)}^2\|\Theta_N (f)\|_{{\rm L}^2(\R^2)}^2+C\int_{\{f<2N\}}f^4 \dd x \\
&\le& C \|\nabla f\|_{{\rm L}^2(\R^2)}^2  \int_{\{f\ge N\}}f^2 \dd x+4CN^2 \int_{\{f<2N\}}f^2 \dd x\\
&\le& \frac{C}{\log N}\|\nabla f\|_{{\rm L}^2(\R^2)}^2 \|f^2 \log  f\|_{{\rm L}^1(\R^2)}+CN^2 \|f\|_{{\rm L}^2(\R^2)}^2\;,
\end{eqnarray*}
hence~\eqref{eqaBHN} by choosing appropriately $N$ in terms of $\eps$.
\end{proof}

\begin{corollary}
 For any $\eps >0$, there exists $L_\eps >0$ such that
  \begin{equation}\label{BHN}
    \|\rho\|_{{\rm L}^2(\R^2)}^2 \le \eps \left\|\frac{\nabla\rho}{\rho}\right\|_{\LL^2(\rho)}^2 \|\rho\log\rho\|_{{\rm L}^1(\R^2)}  + L_\eps \|\rho\|_{{\rm L}^1(\R^2)} \;
  \end{equation}
  for all $\rho \in \LL^1_+(\R^2)$ such that $\rho\log\rho \in \LL^1(\R^2)$
 and ${\nabla\rho}/{\rho}\in\LL^2(\R^2,\rho;\R^2)$.
\end{corollary}

\section{A Carleman type estimate}
\begin{lemma}[Carleman Estimate]
  Let $\rho \in \PP(\RR^2)$ be such that  $\int_{\RR^2} \rho |\log \rho|\dd x $ and $\int_{\RR^2}\rho \log H \dd x$ are finite then
  \begin{equation}\label{eq:carleman}
    \int_{\RR^2} \rho |\log \rho|\dd x \le \int_{\RR^2} \rho \log  \rho\dd x  +\frac{2}{\ee} -2 \int_{\RR^2}\rho \log H\dd x\;.
  \end{equation}
\end{lemma}
\begin{proof}
Set $\bar \rho =\rho\un_{(0,1)}(\rho)$,
\begin{equation*}
  \begin{aligned}
  \int_{\RR^2} \rho |\log \rho|\dd x &= -  \int_{\RR^2} \bar \rho \log \bar \rho\dd x + \int_{\{\rho >1\}} \rho \log  \rho\dd x\\
&= \int_{\RR^2} \rho \log  \rho\dd x - 2 \int_{\RR^2}\bar \rho \log H\dd x -2 \int_{\RR^2}\frac{\bar \rho}{H} \log \left(\frac{\bar \rho}{H}\right)H\dd x
  \end{aligned}
\end{equation*}
Since $\|H\|_1=1$ it follows from Jensen's inequality that
\begin{equation*}
  \begin{aligned}
\int_{\RR^2} \rho |\log \rho|\dd x &\le\int_{\RR^2} \rho \log  \rho\dd x  + 2 \log \pi +4 \int_{\RR^2}\bar \rho \log(1+|x|^2)\dd x\\
&\quad-2 \left(\int_{\RR^2}\frac{\bar \rho}{H} H\dd x \right)\log \left(\int_{\RR^2}\frac{\bar \rho}{H}H\dd x\right)\\
&\le \int_{\RR^2} \rho \log  \rho\dd x  + 2 \log \pi +\frac{2}{\ee} +4 \int_{\RR^2}\bar \rho \log(1+|x|^2)\dd x\;.
  \end{aligned}
\end{equation*}
The desired result comes directly from the definition of $H$ since $\bar \rho \le \rho$.
\end{proof}

\section{Compactness of vector fields}
We recall the following result, see \cite[Theorem 5.4.4]{AGS08}.
\begin{proposition}\label{prop:vf}
Let $\mathcal U$ be an open set of $\R^K$. If $(\mu_n)_n$ is a sequence of probability measures in $\mathcal U$ narrowly converging to $\mu$ and
$(v_n)_n$ is a sequence of vector fields in $\L^2(\mathcal U,\mu_n;\R^K)$
satisfying $$\sup_{n}\|v_n\|_{\L^2(\mathcal U,\mu_n;\R^K)}<+\infty,$$
then there exists a vector field $v\in \L^2(\mathcal U,\mu;\R^K)$
such that
\begin{equation*}
\lim_{n\to\infty}\int_{\mathcal U} \varphi \cdot v_n \dd\mu_n = \int_{\mathcal U} \varphi \cdot v \dd\mu \qquad \varphi\in \C^\infty_0(\mathcal U,\R^K)
\end{equation*}
and
\begin{equation}\label{VFlsc}
\liminf_{n\to\infty} \|v_n\|_{\L^2(\mathcal U,\mu_n;\R^K)} \geq \|v\|_{\L^2(\mathcal U,\mu;\R^K)} .
\end{equation}
\end{proposition}

\bigskip
\noindent{\bf Acknowledgement \/}
AB and MK acknowledge support from ECOS--C11E07 project ``Functional inequalities, asymptotics and dynamics of fronts''.

JAC acknowledges support from projects MTM2011-27739-C04-02 (Feder), 2009-SGR-345 from Ag\`encia de Gesti\'o d'Ajuts Universitaris i de Recerca-Generalitat  de Catalunya, the Royal Society through a Wolfson Research Merit Award, and the Engineering and Physical Sciences Research Council (UK) grant number EP/K008404/1.

SL acknowledges support from a MIUR-PRIN 2010-2011 grant for the project Calculus of Variations. 
SL acknowledges support of the GNAMPA project 'EQUAZIONI DI EVOLUZIONE CON TERMINI NON LOCALI' of the Istituto Nazionale di Alta Matematica (INdAM).

AB, JAC, PL, SL  acknowledges the Centro di Ciencias Pedro Pasqual de Benasque where this work was partially done.

DK is partially supported by NSF DMS 0806703, DMS 0635983 and OISE 0967140. 

MK is partially supported by Chilean research grants Fondecyt 1130126,  Fondo Basal CMM-Chile, and part of this work was done during his stay in the Universit\'e Toulouse 1 Capitole.

\bibliographystyle{siam}
\bibliography{biblio}
\end{document}